\documentclass[12pt]{article}
\usepackage[dvipdfm, letterpaper, margin=2cm]{geometry}
\usepackage{amsmath}%
\usepackage{amsthm}
\usepackage{amssymb}%
\usepackage{amsfonts}
\usepackage{color}
\usepackage{mathrsfs}
\allowdisplaybreaks[4]
\numberwithin{equation}{section}
\newtheorem{theorem}{Theorem}[section]
\newtheorem{proposition}[theorem]{Proposition}
\newtheorem{corollary}[theorem]{Corollary}
\newtheorem{conjecture}[theorem]{Conjecture}
\newtheorem{lemma}[theorem]{Lemma}
\theoremstyle{definition}

\newtheorem{definition}[theorem]{Definition}
\DeclareMathOperator{\ad}{ad}

\DeclareMathOperator{\End}{End}

\DeclareMathOperator{\Res}{Res}

\newcommand{\BF}{\mathbb{F}}
\newcommand{\BN}{\mathbb{N}}
\newcommand{\BZ}{\mathbb{Z}}

\newcommand{\BC}{\mathbb{C}}
\newcommand{\spanf}{\mathrm{span}_\BF}
\newcommand{\fg}{\mathfrak{g}}

\newcommand{\D}{\mathcal{D}}
\newcommand{\B}{\mathcal{B}}

\newcommand{\bc}{\mathbf{c}}
\newcommand{\bk}{\mathbf{k}}
\newcommand{\Vir}{{\mathcal{V} \rm ir}}
\newcommand{\CL}{\mathcal{L}}
\newcommand{\1}{\mathbf{1}}

\begin{document}
\title{Modular Virasoro vertex algebras and affine vertex algebras}
\author{Xiangyu Jiao\footnote{Partially supported by China NSF grants ($\#$11401213, 11571391),
STCSM grant ($\#$13dz2260400)}\\
Department of Mathematics, East China Norma University,\\
Shanghai 200241, China\\
\textit{E-mail address:} \texttt{xyjiao@math.ecnu.edu.cn}\and
Haisheng Li\footnote{Partially supported by China NSF grants  ($\#$11471268, 11571391, 11671247)}\\
Department of Mathematical Sciences, Rutgers University,\\
Camden, NJ 08102, USA, and\\
School of  Mathematical Sciences, Xiamen University, Xiamen 361005, China\\
\textit{E-mail address:} \texttt{hli@camden.rutgers.edu}\and
Qiang Mu\footnote{Partially supported by China NSF grant ($\#$11571391),
Heilongjiang Provincial Natural Science Foundation grant ($\#$LC2015001),
and Program for NCET in Heilongjiang Provincial University grant ($\#$1254-{}-NCET-{}-010)} \thanks{Corresponding author.}\\
School of Mathematical Sciences, Harbin Normal University,\\
Harbin, Heilongjiang 150025, China\\
\textit{E-mail address:} \texttt{qmu520@gmail.com}}
\maketitle

\begin{abstract}
In this paper, we study Virasoro vertex algebras and affine vertex algebras over a general field of characteristic $p>2$.
More specifically, we study certain quotients of the universal Virasoro and affine vertex algebras
by ideals related to the $p$-centers of the Virasoro algebra and affine Lie algebras.
Among the main results, we classify their irreducible $\BN$-graded modules by explicitly determining their Zhu algebras
and show that these vertex algebras have only finitely many irreducible $\BN$-graded modules and they are  $C_2$-cofinite.
\end{abstract}

\section{Introduction}
Vertex (operator) algebra has been extensively studied for about three decades since it was introduced
in mathematics in the late 1980's (see \cite{B86}, \cite{FLM}, \cite{FHL}).
Vertex algebras are a class of highly nonclassical algebraic structures and
they have deep connections with numerous  fields in mathematics and in physics as well.
Among the important examples of vertex  algebras are those associated to
infinite-dimensional Lie algebras such as affine Lie algebras, Heisenberg algebras, and the Virasoro algebra
(see \cite{FZ}), and those associated to non-degenerate even lattices (see  \cite{B86}, \cite{FLM}).

In the past, studies were mainly focused on vertex algebras over fields of characteristic $0$.
(Note that the notion of vertex algebra, introduced the first time by Borcherds in \cite{B86},  is over an arbitrary field.)
In literature, integral forms of vertex algebras have been studied
from various points of view (see \cite{BR1,BR2}, \cite{DG}, \cite{Mc1,Mc2}).
There are also some works on modular vertex algebras and their modules.
In \cite{DR1}, Dong and Ren studied the representations for a general vertex algebra over an
arbitrary field, and in \cite{DR2} they studied modular vertex operator algebras associated to
the Virasoro algebra and their representations with the main focus on the case of central charge $\frac{1}{2}$.
Based on the work of Dong and Griess (see \cite{DG}),
the third named author of this current paper studied  (see \cite{Mu})
vertex algebras over fields of prime characteristic,
by using  integral forms of lattice vertex operator algebras.
In \cite{LM}, two of us studied modular Heisenberg vertex algebras.
Arakawa and Wang studied modular affine vertex algebras at critical levels in \cite{AW}.
Recently, Wang did another study on modular representations for affine ${\mathfrak{sl}}_{2}$ and the Virasoro algebra (see \cite{wang}).

In the case of prime characteristic, most of the conceptual results in characteristic zero  essentially hold true.
For example, a prime-characteristic version of Zhu's $A(V)$-theory was given in \cite{DR1},
while a prime-characteristic version of the conceptual construction of vertex algebras and modules (see \cite{Li96}),
and a prime-characteristic version of the existence theorem of Frenkel-Kac-Ratiu-Wang (see \cite{FKRW}) and Meurman-Primc (see \cite{MP}) were given in \cite{LM}. On the other hand, for concrete examples there are interesting new features.
For example, modular Heisenberg vertex algebras for nonzero levels are no longer simple due to certain central elements
and the simple quotient vertex algebras have only one irreducible module up to equivalences.  This
is closely related to restricted Lie algebras.

In the theory of Lie algebras over a field of prime characteristic $p$,
an important notion is that of restricted Lie algebra  (see \cite{Jac, Jac2}).
For a restricted Lie algebra $\fg$, which is a Lie algebra with a mapping $a\mapsto a^{[p]}$ satisfying certain properties,
a classical fact is that elements $a^p-a^{[p]}$ for $a\in\fg$ lie in the center of
the universal enveloping algebra $U(\fg)$. The subalgebra generated
by these central elements is called the $p$-center, while
the quotient algebra  $\mathfrak{u}(\fg)$ of  $U(\fg)$ by the ideal generated by  $a^p-a^{[p]}$ for all $a\in\fg$ is called
the restricted universal enveloping algebra. If $\fg$ is a finite-dimensional restricted Lie algebra,
it was known that $\mathfrak{u}(\fg)$ is finite-dimensional and it admits only finitely many
non-isomorphic irreducible modules.

In this paper, as a natural continuation of \cite{LM}, we study modular vertex algebras associated to
 the Virasoro algebra and affine Lie algebras.
Let  $V_{\Vir}(c,0)$ be the vertex algebra based on
the generalized Verma module of central charge $c\in \BF$ for the Virasoro algebra $\Vir$
and let $V_{\hat{\fg}}(\ell,0)$ be the vertex algebra based on the generalized Verma module of level $\ell\in \BF$
for an affine Lie algebra $\hat{\fg}$ with $\fg$ a restricted Lie algebra.
In this paper, we study a special quotient vertex algebra $V^{0}_{\Vir}(c,0)$ of $V_{\Vir}(c,0)$
and a special quotient vertex algebra $V^{0}_{\hat{\fg}}(\ell,0)$ of $V_{\hat{\fg}}(\ell,0)$,
which to a certain extent are analogous to restricted universal enveloping algebras.
It is proved that the Zhu algebra of $V^{0}_{\Vir}(c,0)$ is finite-dimensional
and $V^0_{\Vir}(c,0)$ has exactly $p$ irreducible $\BN$-graded modules up to equivalences.
As for vertex algebra $V^{0}_{\hat{\fg}}(\ell,0)$, it is proved that the Zhu algebra  is naturally isomorphic to
the restricted universal enveloping algebra $\mathfrak{u}(\fg)$.
Furthermore, it is shown that if $\fg$ is finite-dimensional, $V^{0}_{\hat{\fg}}(\ell,0)$ has only finitely
 many irreducible $\BN$-graded modules up equivalences.
It is also proved in this paper that vertex algebras $V_{\Vir}^{0}(c,0)$ and $V^{0}_{\hat{\fg}}(\ell,0)$ are
$C_2$-cofinite for any $c,\ell\in \BF$.

We now give a more detailed count of the contents of this paper.
Throughout this paper, fix a field $\BF$ of characteristic $p>2$.
We start with the Virasoro algebra, which, denoted by $\Vir$,  is a Lie algebra with a basis
$\{ L_{n}\mid n\in \BZ\}\cup \{ \bc\}$, where $\bc$ is central and
\begin{equation*}
    [L_{m},L_{n}]=(m-n)L_{m+n}+\frac{1}{2}\binom{m+1}{3}\bc
\end{equation*}
for $m,n\in \BZ$. (Note that this defining relation is well defined for $p> 2$.)
First, we show that $\Vir$ has a restricted Lie algebra structure with
$\bc^{[p]}=\bc$ and for $n\in \BZ$,
\begin{equation}
(L_{n})^{[p]}=\begin{cases}L_{np}& \   \   \mbox{if }p\mid n\\
0& \   \   \mbox{if }p\nmid n.
\end{cases}
\end{equation}
For each $c\in \BF$, let $V_{\Vir}(c,0)$ be the universal vacuum $\Vir$-module of central charge $c$,
which is defined to be the $\Vir$-module with a generator $\1$, subject to relations
\begin{equation*}
    \bc \1=c\1,\    \    \    \   L_{n}\1=0\   \   \mbox{ for }n\ge -1.
\end{equation*}
Just as in the case of characteristic $0$ (cf. \cite{FZ}), one can show (see \cite{LM}) that
$V_{\Vir}(c,0)$ has a canonical vertex algebra structure with $\1$ as the vacuum vector.
Vertex algebra $V_{\Vir}(c,0)$ has a natural $\BZ$-grading which is uniquely determined by
$\deg \1=0$ and $\deg L_{n}=-n$ for $n\in \BZ$. Let $I_{0}$ be the $\Vir$-submodule of $V_{\Vir}(c,0)$ generated by vectors
\begin{equation*}
    (L(-n)^{p}-L(-n)^{[p]})\1\    \    \    \mbox{ for }n\ge 2.
\end{equation*}
It is proved that $I_{0}$ is an ideal of the vertex algebra $V_{\Vir}(c,0)$. Denote by $V^{0}_{\Vir}(c,0)$ the quotient vertex algebra of
$V_{\Vir}(c,0)$ by $I_{0}$. Then it is proved that the Zhu algebra of $V^{0}_{\Vir}(c,0)$ is isomorphic to the algebra
$\BF[x]/(x^p-x)\BF[x]$.
It follows from Zhu's $A(V)$-theory that $V^0_{\Vir}(c,0)$ has exactly $p$ irreducible $\BN$-graded modules up to equivalences.

Let $\fg$ be a Lie algebra equipped with a non-degenerate symmetric invariant bilinear form $\langle \cdot,\cdot \rangle$.
Denote by $\hat{\fg}$ the affine Lie algebra.
Just as in the case of characteristic $0$ (cf. \cite{FZ}), for each $\ell\in \BF$, the universal vacuum module $V_{\hat{\fg}}(\ell,0)$
 for the affine Lie algebra $\hat{\fg}$ of level $\ell$ has a canonical vertex algebra structure and
it is naturally $\BZ$-graded.
Assume that $\fg$ is a restricted Lie algebra. It was known (see \cite{Ma}) that  $\hat{\fg}$ has
a restricted Lie algebra structure with $a(n)^{[p]}=a^{[p]}(np)$ for $n\in \BZ$ and $\bk^{[p]}=\bk$.
Let $J_{0}$ be the $\hat{\fg}$-submodule of $V_{\hat{\fg}}(\ell,0)$, generated by vectors
\begin{equation*}
    (a(-n)^{p}-a^{[p]}(-np))\1\    \    \   \mbox{ for }a\in \fg,\  n\ge 1.
\end{equation*}
It is proved that $J_{0}$ is a graded ideal of the vertex algebra $V_{\hat{\fg}}(\ell,0)$.
Denote by $V^{0}_{\hat{\fg}}(\ell,0)$ the quotient vertex algebra of $V_{\hat{\fg}}(\ell,0)$ by $J_{0}$.
It is proved that the Zhu algebra $A(V^{0}_{\hat{\fg}}(\ell,0))$ is naturally isomorphic to the
restricted universal enveloping algebra $\mathfrak{u}(\fg)$. It then follows that  irreducible $\BN$-graded
$V^{0}_{\hat{\fg}}(\ell,0)$-modules are classified by irreducible
$\mathfrak{u}(\fg)$-modules. Note that $\mathfrak{u}(\fg)$ is finite-dimensional if $\fg$ is finite-dimensional.
Thus $V^{0}_{\hat{\fg}}(\ell,0)$ has only finitely many irreducible
$\BN$-graded modules up to equivalences, provided that $\fg$ is a finite-dimensional restricted Lie algebra.

In vertex operator algebra theory, an important concept is $C_2$-cofiniteness due to Zhu.
In this paper, we also study the $C_2$-cofiniteness  of vertex algebras $V_{\Vir}^{0}(c,0)$ and $V^{0}_{\hat{\fg}}(\ell,0)$.
For a general vertex algebra $V$, set
\begin{equation*}
    C_{2}(V)={\rm span}\{  u_{-2-k}v\mid  u,v\in V,\ k\in \BN\}.
\end{equation*}
(Note that in the case of characteristic $0$, this definition agrees with Zhu's definition in \cite{zhu1, zhu2}.)
Just as in the case of characteristic $0$, the quotient space $V/C_2(V)$ has a canonical commutative and associative algebra structure.
(In fact, it has a Poisson algebra structure.) A vertex algebra $V$ is said to be $C_2$-cofinite if $V/C_2(V)$ is finite-dimensional.
It is proved that vertex algebras $V_{\Vir}^{0}(c,0)$ and $V^{0}_{\hat{\fg}}(\ell,0)$ are $C_2$-cofinite
for any $c,\ell\in \BF$.

Note that  $V_{\Vir}(c,0)$ as a $\Vir$-module has a (unique) maximal graded $\Vir$-submodule $J$.
It is proved that $J$ is an ideal of vertex algebra $V_{\Vir}(c,0)$.
Then the quotient $L_{\Vir}(c,0)$ of $V_{\Vir}(c,0)$ by $J$ is a simple graded vertex algebra.
Similarly, one has a simple graded vertex algebra $L_{\hat{\fg}}(\ell,0)$, which is
the quotient of $V_{\hat{\fg}}(\ell,0)$ modulo the maximal
graded $\hat{\fg}$-submodule. In the case of characteristic $0$,  irreducible modules for $L_{\Vir}(c,0)$ and
$L_{\hat{\fg}}(\ell,0)$ are the main focus.
Certainly, classifying  irreducible $\BN$-graded modules for
 modular vertex algebras $L_{\Vir}(c,0)$ and $L_{\hat{\fg}}(\ell,0)$ will be an important problem.
 We shall study this problem in a future publication.

This paper is organized as follows: Section 2 is preliminary; We review basic notations and notions, and we recall some
results on vertex algebras and their modules, Zhu's $A(V)$-theory, $C_2$-cofiniteness.
In Section 3, we study the Virasoro vertex operator algebra $V_{\Vir}(c,0)$.
In Section 4, we determine the Zhu algebra $A(V^{0}_{\Vir}(c,0))$ and classify irreducible $\BN$-graded
$V^{0}_{\Vir}(c,0)$-modules.
In Section 5, we study affine vertex algebras $V_{\hat{\fg}}(\ell,0)$.
In Section 6, we determine the Zhu algebra of vertex algebra $V^{0}_{\hat{\fg}}(\ell,0)$ and classify
irreducible $\BN$-graded $V^{0}_{\hat{\fg}}(\ell,0)$-modules. In the Appendix, we prove a simple fact
on binomial coefficients.

\section{Preliminaries}

This section is mostly expository.
In this section, we recall basic definitions of various notions, including vertex algebras  and modules,
Zhu's $A(V)$-theory from \cite{DR1}, $C_2$-cofiniteness. We also recall some technical results from \cite{LM}.

First of all, let $\BF$ be an algebraically closed field of characteristic $p\ge 3$, which is fixed throughout this paper.
All vector spaces in this paper are considered to be over $\BF$.
We use the usual symbols $\BZ$ for the integers,
$\BZ_+$ for the positive integers, and $\BN$ for the nonnegative integers. Note that for any $m\in \BZ,\ n\in \BN$,
we have
\begin{equation*}
    \binom{m}{n}=\frac{m(m-1)\cdots (m+1-n)}{n!}\in \BZ,
\end{equation*}
so that $\binom{m}{n}$ is a well defined element of the field $\BF$.

We shall use the formal variable notations and conventions as established in \cite{FLM} and \cite{FHL}.
In particular, symbols $x, z, x_0,x_1,x_2$ will be mutually commuting independent formal variables, and
given a vector space $U$, we have vector spaces
\begin{eqnarray*}
U[[x,x^{-1}]]&=&\left\{  \sum_{n\in \BZ}u(n)x^{n}\Biggm| u(n)\in U\right\},\\
U((x))&=&\left\{ \sum_{n\ge k}u(n)x^{n}\Biggm| k\in \BZ,\ u(n)\in U\right\}.
\end{eqnarray*}
The formal delta function is
\begin{equation*}
    \delta(x)=\sum_{n\in \BZ}x^{n}\in \BF[[x,x^{-1}]].
\end{equation*}
We shall also use the formal variable expansion convention
\begin{equation}
(x_1\pm x_2)^{n}=\sum_{i\ge 0}\binom{n}{i}(\pm 1)^{i}x_1^{n-i}x_2^{i}
\end{equation}
for any $n\in \BZ$. Furthermore, we have
\begin{equation*}
    x_0^{-1}\delta\left(\frac{x_1\pm x_2}{x_0}\right)=\sum_{n\in \BZ}x_0^{-n-1}(x_1\pm x_2)^{n}.
\end{equation*}

In the theory of modular Lie algebras, an important role is played by the {\em Hasse derivatives}, which are defined by
\begin{equation}
 \partial_x^{(n)} x^m=\binom{m}{n}x^{m-n}\    \    \  \mbox{ for }m\in\BZ, \ n\in \BN.
\end{equation}
More generally,  we view $\partial_x^{(n)}$ with $n\in \BN$ as linear operators on $U[[x,x^{-1}]]$ for any vector space $U$.
These operators satisfy relations
\begin{eqnarray}
   \partial_x^{(m)} \cdot  \partial_x^{(n)} =\binom{m+n}{n}\partial_x^{(m+n)}
 \end{eqnarray}
 for $m,n\in \BN$.

 Now, define an (abstract) bialgebra $\B$  (cf. \cite{B86}) with a designated basis $\{\D^{(n)}\mid n\in\BN\}$, where
\begin{eqnarray}
    && \D^{(m)}\cdot \D^{(n)}=\binom{m+n}{n}\D^{(m+n)}, \quad
    \D^{(0)}=1,\label{eD-multiplication}\\
    && \Delta(\D^{(n)})=\sum_{i=0}^n \D^{(n-i)}\otimes \D^{(i)}, \quad
    \varepsilon(\D^{(n)})=\delta_{n,0}\label{ecoproduct-D}
\end{eqnarray}
for $m,n\in\BN$. Then, for any vector space $U$, $U[[x,x^{-1}]]$ is naturally a $\B$-module with
$\D^{(n)}$ acting as $\partial_x^{(n)}$ for $n\in \BN$. Following \cite{DR1}, we set
\begin{equation}
    e^{x\D}=\sum_{n\ge0}x^n \D^{(n)}\in \B[[x]],
\end{equation}
so that $e^{x\D}$ acts on every $\B$-module.

Recall that a {\em $\B$-module algebra} is an associative algebra $A$ equipped with a $\B$-module
structure such that
\begin{equation*}
    b\cdot(uv)=\sum (b^{(1)}\cdot u)(b^{(2)}\cdot v)
\end{equation*}
for $b\in \B,\ u,v\in A$, where $\Delta(b)=\sum b^{(1)}\otimes b^{(2)}$, whereas
a {\em $\B$-module Lie algebra} is a Lie algebra $\fg$ equipped with a $\B$-module structure such that
\begin{equation*}
    b\cdot[u,v]=\sum [b^{(1)}u,b^{(2)}v]
\end{equation*}
for $b\in \B,\ u,v\in \fg$. A basic fact is that if $\fg$ is a $\B$-module Lie algebra, then the universal enveloping algebra
$U(\fg)$ is naturally a $\B$-module algebra.

Note that Borcherds in \cite{B86} defined the notion of vertex algebra for an arbitrary field.
Just  as in the case of characteristic zero,  the notion of vertex algebra over $\BF$ can be  defined alternatively
by using the vacuum property, creation property, and the Jacobi identity in terms of generating functions as follows (cf. \cite{LL}):

\begin{definition}\label{def-va}
A {\em vertex algebra} is a vector space
$V$ over $\BF$, equipped with a linear map
\begin{align*}
Y(\cdot,x):\ \ V&\rightarrow (\End V)[[x,x^{-1}]] \\
v&\mapsto Y(v,x)=\sum_{n\in \BZ}v_{n}x^{-n-1}\   \  (\mbox{where }v_{n}\in \End V)
\end{align*}
and a vector $\1\in V$ such that
\begin{gather*}
    Y(u,x)v\in V((x))\   \   \mbox{ for }u,v\in V,\\
    Y(\1,x)v=v,\  \   Y(v,x)\1\in V[[x]]\   \mbox{ and }\ \lim_{x\rightarrow 0}Y(v,x)\1=v\   \  \mbox{ for }v\in V,
\end{gather*}
and such that for $u,v,w\in W$,
\begin{eqnarray}\label{Jacobi-va}
&&x_0^{-1}\delta\left(\frac{x_1-x_2}{x_0}\right)Y(u,x_1)Y(v,x_2)w-x_0^{-1}\delta\left(\frac{x_2-x_1}{-x_0}\right)Y(v,x_2)Y(u,x_1)w
\nonumber\\
&&\hspace{3cm}=x_2^{-1}\delta\left(\frac{x_1-x_0}{x_2}\right)Y(Y(u,x_0)v,x_2)w.
\end{eqnarray}
\end{definition}

Let $V$ be a vertex algebra. It is known (see \cite{B86}; cf. \cite{LM}) that  $V$ is a $\B$-module with
\begin{equation}\label{Dn-vacuum}
\D^{(n)}v=v_{-n-1}\1\   \    \    \mbox{ for }n\in \BN,\  v\in V,
\end{equation}
which can be written as
\begin{equation}
e^{x\D}v=Y(v,x)\1\   \   \   \mbox{ for }v\in V.
\end{equation}
The last property is a special case of the  {\em skew symmetry}
\begin{equation}
Y(u,x)v=e^{x\D}Y(v,-x)u\    \    \    \mbox{ for }u,v\in V.
\end{equation}
We also have
\begin{equation}\label{conjugation}
e^{z\D}Y(v,x)e^{-z\D}=Y(e^{z\D}v,x)=e^{z\partial_{x}}Y(v,x)
\end{equation}
for $v\in V$, where $e^{z\partial_{x}}=\sum_{n\in\BN}z^{n}\partial_{x}^{(n)}$.

For a vertex algebra $V$, the notions of left ideal, right ideal, and ideal are defined in the usual way:
An {\em ideal} of $V$ is a subspace $I$ such that
\begin{equation}
v_{n}w,\   \   w_{n}v\in I\   \   \   \mbox{ for all }v\in V,\ n\in \BZ,\ w\in I,
\end{equation}
whereas a {\em left (right) ideal} of $V$ is a subspace $I$ such that $v_{n}w\in I$ ($w_{n}v\in I$) for all $v\in V,\ n\in \BZ,\ w\in I$.
In view of the skew symmetry, an ideal is the same as a $\B$-stable left (or right) ideal.
On the other hand, in view of (\ref{Dn-vacuum}),
a right ideal is automatically $\B$-stable, so that a right ideal is just the same as an ideal.

\begin{definition}\label{va-module}
For a vertex algebra $V$, a {\em $V$-module} is a vector space $W$ equipped with a linear map
\begin{align*}
Y_{W}(\cdot,x):\ \ V&\rightarrow (\End W)[[x,x^{-1}]] \\
v&\mapsto Y_{W}(v,x),
\end{align*}
satisfying the conditions that $Y_{W}(\1,x)=1_{W}$ (the identity operator on $W$),
\begin{equation*}
    Y_{W}(u,x)w\in W((x))\   \   \   \mbox{ for }u\in V,\ w\in W,
\end{equation*}
and that for $u,v\in V,\ w\in W$, the Jacobi identity (\ref{Jacobi-va}) with the obvious changes holds.
\end{definition}

Note that a vertex algebra $V$ is naturally a $V$-module and a left ideal of vertex algebra $V$ amounts to
a $V$-submodule of the {\em adjoint} $V$-module $V$.

We have (see \cite{LM}):

\begin{lemma}\label{D-module}
Let $(W,Y_{W})$ be a module for a vertex algebra $V$. Then
\begin{equation}
Y_{W}(\D^{(n)}v,x)=\partial_{x}^{(n)}Y_{W}(v,x)
\end{equation}
for $n\in \BN,\ v\in V$.
\end{lemma}

The following result was obtained in \cite{LM} (cf. \cite{AW}):

\begin{lemma}\label{th:a(x)^p}
Let $V$ be a vertex algebra and $(W,Y_{W})$ a $V$-module.
Then
\begin{equation*}
    Y_{W}\left((a_{-n})^p\1,x\right)=\Biggl(\sum_{j\ge0}\binom{-n}{j}(-1)^j a_{-n-j}x^{j}\Biggr)^p
            +\Biggl(-\sum_{j\ge0}\binom{-n}{j}(-1)^{n+j}a_{j}x^{-n-j}\Biggr)^p
\end{equation*}
for $a\in V,\  n\in\BZ_+$.
\end{lemma}

Furthermore, we have:

\begin{corollary}\label{th:a_n^p}
In the setting of Lemma \ref{th:a(x)^p}, in addition
assume that $a$ is an element of $V$ satisfying the condition that $a_{i}a\in \BF \1$ for all $i\in \BN$.
 Then
\begin{equation*}
    Y_{W}\left((a_{-n})^p \1,x\right)=\sum_{j\ge0}\binom{n+j-1}{j} (a_{-n-j})^p x^{j p}
            +\sum_{j\ge0}(-1)^{n-1}\binom{n+j-1}{j} (a_{j})^p x^{(-n-j)p}
\end{equation*}
for $n\in\BZ_+$.  In particular,
\begin{equation*}
    Y_{W}\left((a_{-1})^p \1,x\right)=\sum_{j\in\BZ} (a_{j})^p x^{(-j-1)p}.
\end{equation*}
\end{corollary}

\begin{proof} Recall Borcherds' commutator formula
\begin{equation*}
    [a_{m},a_{n}]=\sum_{i\ge 0}\binom{m}{i}(a_ia)_{m+n-i}
\end{equation*}
for $m,n\in \BZ$. As $a_ia\in \BF \1$ for $i\in \BN$ and $\1_{k}=\delta_{k,-1}$ for $k\in \BZ$, we have that
$[a_{m},a_{n}]=0$ for $m,n\ge 0$ and for $m,n<0$. Then it follows from Lemma \ref{th:a(x)^p}
(and the Fermat's little theorem) immediately.
\end{proof}

\begin{definition}\label{dgradedva}
A vertex algebra $V$ equipped with a $\BZ$-grading  $V=\bigoplus_{n\in \BZ}V_{(n)}$ is called a
{\em $\BZ$-graded vertex algebra} if  $\1\in V_{(0)}$ and
\begin{eqnarray}
u_{k}V_{(n)}\subset V_{(m+n-k-1)}
\end{eqnarray}
for $u\in V_{(m)},\ m,n,k\in \BZ$.
\end{definition}

Let $V=\bigoplus_{n\in \BZ}V_{(n)}$ be a $\BZ$-graded vertex algebra.
Recall that $\D^{(k)}v=v_{-k-1}\1$ for $k\in \BN,\ v\in V$. As an immediate consequence,
we have
\begin{equation}
\D^{(k)}V_{(n)}\subset V_{(n+k)}\   \    \   \mbox{ for }k\in \BN,\ n\in \BZ.
\end{equation}

For a $\BZ$-graded vertex algebra $V$, we define a {\em $\BZ$-graded $V$-module}
to be a $V$-module $(W,Y_{W})$ equipped with a $\BZ$-grading $W=\bigoplus_{n\in \BZ}W(n)$,
satisfying the condition
\begin{eqnarray}
u_{m}W(n)\subset W(k+n-m-1)
\end{eqnarray}
for $u\in V_{(k)},\ k,m,n\in \BZ$. Two $\BZ$-graded $V$-modules $W_1$ and $W_2$ are said to be {\em equivalent} if there
exist a $V$-module isomorphism $\psi: W_1\rightarrow W_2$ and an integer $k$ such that
\begin{equation*}
    \psi (W_1(n))=W_2(n+k)\   \   \   \mbox{ for all }n\in \BZ.
\end{equation*}
An {\em $\BN$-graded $V$-module} is a $\BZ$-graded $V$-module $W=\bigoplus_{n\in \BZ}W(n)$
with $W(n)=0$ for $n<0$. {\em Equivalence} between two $\BN$-graded $V$-modules is defined
correspondingly. Then each nonzero $\BN$-graded $V$-module is equivalent to an $\BN$-graded $V$-module
$W$ with $W(0)\ne 0$.

The following result is about the existence of maximal ideals:

\begin{lemma}\label{max-ideal}
Let $V=\bigoplus_{n\in \BZ}V_{(n)}$ be a $\BZ$-graded vertex algebra such that $V_{(n)}=0$ for $n<0$ and $V_{(0)}=\BF \1$.
Let $J$ be the maximal graded $V$-submodule of $V$ with $J_{(0)}=0$. Then $J$ is a maximal graded ideal.
\end{lemma}

\begin{proof} To prove that $J$ is an ideal, we prove that $J$ is $\B$-stable.
Set
\begin{equation*}
    \bar{J}={\rm span}\{ \D^{(n)}v\mid  n\in \BN,\ v\in J\}.
\end{equation*}
Let $u\in V, \  v\in J,\ n\in \BN,\ m\in \BZ$. From (\ref{conjugation}) we have
\begin{eqnarray*}
\D^{(n)}(u_{m}v)=\sum_{i=0}^{n}(-1)^{i}\binom{m}{i}u_{m-i}\D^{(n-i)}v
=u_{m}\D^{(n)}v+\sum_{i=1}^{n}(-1)^{i}\binom{m}{i}u_{m-i}\D^{(n-i)}v,
\end{eqnarray*}
which gives
\begin{eqnarray}\label{proof-need-1}
u_{m}\D^{(n)}v=\D^{(n)}(u_{m}v)-\sum_{i=1}^{n}(-1)^{i}\binom{m}{i}u_{m-i}\D^{(n-i)}v.
\end{eqnarray}
Notice that $u_{m}v\in J$ as $J$ is a left ideal.
It follows from induction on $n$ and (\ref{proof-need-1}) that
$u_{m}\D^{(n)}v\in \bar{J}$ for all $m\in \BZ,\ n\in \BN$. This proves that $\bar{J}$ is a left ideal.
Recall that  $\D^{(k)}V_{(n)}\subset V_{(n+k)}$ for $k,n\in \BN$.
Then $\bar{J}$ is also graded.
Since $J_{(0)}=0$, we have
\begin{equation*}
    \D^{(n)}J\subset \bigoplus_{k\ge 1}V_{(k)}\    \    \mbox{ for }n\in \BN,
\end{equation*}
which implies $\bar{J}_{(0)}=0$.
It follows from the definition of $J$ that $\bar{J}\subset J$.
This proves that $J$ is $\B$-stable. Therefore, $J$ is an ideal of $V$.
As $V_{(0)}=\BF \1$, it can be readily seen that $J$ is the (unique) maximal graded ideal.
\end{proof}

Next, we recall the Zhu algebra associated to a $\BZ$-graded vertex algebra.
Let $V=\bigoplus_{n\in \BZ}V_{(n)}$ be a general $\BZ$-graded vertex  algebra. If $u\in V_{(n)}$ with $n\in \BZ$, we say
$u$ is {\em homogeneous of degree $n$} and we write
\begin{equation}
\deg u=n.
\end{equation}
Following Zhu (see \cite{zhu1,zhu2}, \cite{lian}),
for $a,b\in V$ with $a$ homogeneous, define
\begin{eqnarray}\label{a*b}
a*b=\Res_{z}\frac{(1+z)^{\deg a}}{z}Y(a,z)b.
\end{eqnarray}
Then define $a*b$ for general vectors $a,b\in V$ by linearity. This gives us a nonassociative algebra $(V,*)$.
On the other hand, following Dong-Ren (see  \cite{DR1}),  for $a,b\in V$ with $a$ homogeneous and for $n\in\BN$,  set
\begin{align*}
    a\circ_n b&=\Res_z \frac{(1+z)^{\deg a}}{z^{n+2}}Y(a,z)b\\
    &\phantom{:} =\sum_{i\ge 0}\binom{\deg a}{i}a_{i-n-2}b.
\end{align*}
Then denote by $O(V)$ the linear span of $a\circ_n b$ for $a,b\in V$, $n\in\BN$.

The following result was obtained in  \cite{DR1}:

\begin{proposition}\label{zhu-dong-ren}
The subspace $O(V)$ is an ideal of the nonassociative algebra $(V,*)$ and the quotient $V/O(V)$
is an associative algebra, denoted by $A(V)$, where $\1+O(V)$ is the identity.
\end{proposition}

For $u\in V$, set
\begin{eqnarray}
[u]=u+O(V)\in A(V).
\end{eqnarray}
For $u,v\in V$, write the product of $[u]$ and $[v]$ in $A(V)$ as $[u][v]$. Then  $[u*v]=[u][v]$.

The following is a simple fact:

\begin{lemma}\label{th:DnvmodO}
Let $V$ be a $\BZ$-graded vertex  algebra and $v$ any homogeneous vector in $V$. Then
\begin{equation*}
    \D^{(k)}v\equiv \binom{-\deg v}{k}v \pmod{O(V)}\   \   \   \mbox{  for }k\in \BN.
\end{equation*}
In particular, we have $\D^{(1)}v\equiv (-\deg v)v\pmod{O(V)}$.
\end{lemma}

\begin{proof}
We use induction on $k$. It is clear for $k=0$. Assume $k\ge 1$. We have
\begin{align*}
    v\circ_{k-1}\1=\sum_{i\ge0}\binom{\deg v}{i}v_{i-k-1}\1
    =\sum_{i=0}^{k}\binom{\deg v}{i}\D^{(k-i)}v,
\end{align*}
which implies
\begin{align*}\label{need}
\sum_{i=0}^{k}\binom{\deg v}{i}\D^{(k-i)}v \in O(V).
\end{align*}
Using this and the induction hypothesis, we get
\begin{align*}
    \D^{(k)}v\equiv -\sum_{i=1}^{k}\binom{\deg v}{i}\D^{(k-i)}v
    \equiv -\sum_{i=1}^{k}\binom{\deg v}{i}\binom{-\deg v}{k-i}v
    =\binom{-\deg v}{k}v \pmod{O(V)},
\end{align*}
noticing that
\begin{align*}
    \sum_{i=0}^{k}\binom{\deg v}{i}\binom{-\deg v}{k-i}=\binom{0}{k}=0.
\end{align*}
This completes the induction and the proof.
\end{proof}

Just as in the case of characteristic zero (see \cite{DLM}; cf. \cite{zhu1,zhu2}) we have (see \cite{DR1}):

\begin{proposition}\label{WtoA(V)module}
Let $V$ be a $\BZ$-graded vertex  algebra and $W$ a $V$-module. Set
\begin{equation*}
    \Omega(W)=\{w\in W\mid u_n w=0 \  \mbox{ for homogeneous } u\in V,\  n\ge\deg u\}.
\end{equation*}
Then $\Omega(W)$ is an $A(V)$-module with $[u]$ acting as $u_{n-1}$ for $u\in V_{(n)},\ n\in \BZ$.
Furthermore, if $W=\bigoplus_{n\in \BN}W(n)$ is an $\BN$-graded $V$-module, then $W(0)$ is an $A(V)$-submodule of $\Omega(W)$.
\end{proposition}

On the other hand, we have (see \cite{DR1}):

\begin{proposition}\label{A(V)toV}
Let $U$ be an $A(V)$-module. Then there exists  an $\BN$-graded $V$-module $L(U)=\bigoplus_{n\in \BN}L(U)(n)$
such that $L(U)(0)\cong U$ as an $A(V)$-module. Furthermore, if $U$ is irreducible,
$L(U)$ is an irreducible $\BN$-graded $V$-module.
\end{proposition}

The following is a technical result:

\begin{lemma}\label{singular-ideal}
Let $V$ be a $\BZ$-graded vertex  algebra and
suppose that  $v\in \Omega(V)$ is a homogeneous vector of positive degree.
Denote by $I$ the left ideal of $V$
generated by $\D^{(n)}v$ for all $n\ge 0$.
Then $I$ is a proper ideal of $V$.
\end{lemma}

\begin{proof}  We first show that $I$ is an ideal. For this,
it suffices to show that $I$ is $\B$-stable. By definition, $I$ is the $V$-submodule of $V$
generated by $\D^{(n)}v$ for $n\ge 0$. From a result of \cite{DM}, \cite{li-thesis}, we have
\begin{equation*}
    I={\rm span} \{ u_{m}\D^{(n)}v\mid  u\in V,\ m\in \BZ,\ n\in \BN\}.
\end{equation*}
For $u\in V,\ m\in \BZ$ and for $n,j\ge0$, we have
\begin{align*}\label{temp-need}
    \D^{(j)}(u_m\D^{(n)}v)&=\sum_{i=0}^j (-1)^i \binom{m}{i}u_{m-i}\D^{(j-i)}\D^{(n)}v\nonumber\\
    &=\sum_{i=0}^j (-1)^i \binom{m}{i}\binom{n+j-i}{n}u_{m-i}\D^{(n+j-i)}v \in I.
\end{align*}
Then $I$ is a $\B$-submodule and hence an ideal of $V$.

Next, we prove that $I\ne V$.  For $n\in \BN$, we have $\deg \D^{(n)}v=n+\deg v$.
As $v$ is homogeneous of positive degree, we have
$\D^{(n)}v\in \bigoplus_{k\ge 1}V_{(k)}$.
Let $u$ be any homogeneous element of $V$ and let $m\in \BZ,\ n\in \BN$. Note that $\deg u_{m}=\deg u-m-1$.
If $m\le \deg u-1+n$, we have
\begin{equation*}
    \deg u_{m}\D^{(n)}v=\deg u-m-1+n+\deg v\ge \deg v\ge 1,
\end{equation*}
so that $u_{m}\D^{(n)}v\in \bigoplus_{k\ge 1}V_{(k)}$.
If we can show that $u_{m}\D^{(n)}v=0$ for $m\ge \deg u+n$, then we will have
$I\subset \bigoplus_{k\ge 1}V_{(k)}$, which implies that $I$ is proper.
Now, we use induction on $n$ to show that  $u_{m}\D^{(n)}v=0$
 whenever $m\ge \deg u+n$. It is true for $n=0$ as $v\in  \Omega(V)$. Assume that $n\ge 1$ and
 $u_{m'}\D^{(j)}v=0$ for $0\le j\le n-1$ and for all $m'\ge \deg u+j$. Note that
\begin{eqnarray*}
\D^{(n)}(u_{m}v)=\sum_{i=0}^{n}(-1)^{i}\binom{m}{i}u_{m-i}\D^{(n-i)}v
=u_{m}\D^{(n)}v+\sum_{i=1}^{n}(-1)^{i}\binom{m}{i}u_{m-i}\D^{(n-i)}v.
\end{eqnarray*}
Using this and then using the induction hypothesis we obtain
\begin{eqnarray}
u_{m}\D^{(n)}v=\D^{(n)}(u_{m}v)-\sum_{i=1}^{n}(-1)^{i}\binom{m}{i}u_{m-i}\D^{(n-i)}v=0
\end{eqnarray}
for $m\ge \deg u+n$, noticing that $u_{m}v=0$ as $m\ge \deg u$.
This completes the induction and the whole proof.
\end{proof}

Let $V$ be a general vertex algebra (over $\BF$). Set
\begin{equation}
C_{2}(V)={\rm span}\{ u_{-2-n}v\mid u,v\in V,\ n\in \BN\}.
\end{equation}
Notice that for $v\in V$, we have
\begin{equation}
\D^{(k)}v=v_{-k-1}\1\in C_{2}(V)\    \    \    \mbox{ for }k\ge 1.
\end{equation}
Define a binary operation on $V$ by $u\cdot v=u_{-1}v$ for $u,v\in V$.
The same arguments of Zhu in the case of characteristic $0$ (see \cite{zhu1,zhu2}) show:

\begin{proposition}
The subspace $C_{2}(V)$ is an ideal of the nonassociative algebra $(V,\cdot)$ and
the quotient $V/C_{2}(V)$ is a commutative associative algebra with $\1+C_{2}(V)$ as an identity.
\end{proposition}

As in the case of characteristic $0$ (see \cite{zhu1,zhu2}), we formulate the following notion:

\begin{definition}
 A vertex algebra $V$ is said to be {\em $C_2$-cofinite} if $\dim V/C_2(V)<\infty$.
\end{definition}

The following result is useful in determining the commutative associative algebra $V/C_2(V)$:

\begin{lemma}\label{symmetric-algebra}
Let $V$ be a vertex algebra and let $U$ be a subspace of $V$ such that
\begin{equation}\label{span-V}
V={\rm span} \{ u^{(1)}_{-n_1}\cdots u^{(r)}_{-n_r}\1\mid  r\ge 0,\  u^{(i)}\in U,\  n_i\ge 1\}.
\end{equation}
Then there exists a surjective algebra homomorphism $\pi$ from (the symmetric algebra) $S(U)$ to $V/C_2(V)$
such that  $\pi (u)=u+C_{2}(V)$ for $u\in U$.
\end{lemma}

\begin{proof} For $u\in V$, set $\bar{u}=u+C_2(V)\in V/C_2(V)$.
Since $V/C_2(V)$ is a commutative associative algebra, there exists an algebra homomorphism
$\pi: S(U)\rightarrow V/C_2(V)$ such that $\pi (u)=\bar{u}$ for $u\in U$. Then it remains to show that $\pi$ is onto.
 Let $a,u,v\in V,\ n,k\in \BN$. We have
\begin{equation*}
    a_{-n}u_{-2-k}v=u_{-2-k}a_{-n}v+\sum_{i\ge 0}\binom{-n}{i}(a_iu)_{-2-k-n-i}v\in C_2(V).
\end{equation*}
Thus $a_{-n}C_2(V)\subset C_2(V)$ for $a\in V,\ n\ge 0$.
Combining this with the spanning condition (\ref{span-V}) we get
\begin{equation*}
    V=C_2(V)+{\rm span} \{ u^{(1)}_{-1}\cdots u^{(r)}_{-1}\1\mid  r\ge 0,\  u^{(i)}\in U\},
\end{equation*}
which implies that algebra $V/C_2(V)$ is generated by $\bar{u}$ for $u\in U$.
It then follows that $\pi$ is onto.
\end{proof}

Let $V$ be a vertex algebra.  A {\em $(V,\B)$-module} is a $V$-module $(W,Y_W)$ equipped with
a $\B$-module structure such that
\begin{equation}
    e^{x\D}Y_W(v,z)e^{-x\D}=Y_W(e^{x\D}v,z)\  \  \   \mbox{ for }v\in V.
\end{equation}
On the other hand, let $A$ be a $\B$-module algebra. An {\em $(A,\B)$-module} is a module $W$ for $A$ and $\B$ such that
\begin{equation}
    e^{z\D}ae^{-z\D}w=(e^{z\D}a)w\   \   \  \mbox{ for }a\in A,\ w\in W.
\end{equation}
It can be readily seen that a $\B$-module algebra $A$ itself is an $(A,\B)$-module.

\section{Virasoro vertex algebra $V_{\Vir}(c,0)$ and its quotients}
In this  section, we first recall the Virasoro algebra over $\BF$ and the definition of a vertex operator algebra,
and then we introduce the Virasoro vertex  algebra
$V_{\Vir}(c,0)$ associated to the Virasoro algebra of central charge $c\in \BF$.
Furthermore, we study quotient vertex algebras of the Virasoro vertex operator algebra
$V_{\Vir}(c,0)$ by ideals associated to the $p$-center of the universal enveloping algebra $U(\Vir)$.

The Virasoro algebra over $\BF$, which is denoted by $\Vir$, is a Lie algebra with a basis
$\{L_n\mid n\in\BZ\}\cup\{\bc\}$, where $\bc$ is central and
\begin{equation}\label{Virasoro-relations}
    [L_m,L_n]=(m-n)L_{m+n}+\frac{1}{2}\binom{m+1}{3}\delta_{m+n,0}\bc
\end{equation}
for $m,n\in \BZ$. (Note that $p\ne 2$ by assumption.)

We call a $\Vir$-module $W$ a {\em locally truncated module} if for every $w\in W$,
$L(n)w=0$ for $n$ sufficiently large,
where $L(n)$ denotes the operator on $W$ corresponding to $L_n$.
On the other hand, if $\bc$ acts as a scalar $c\in\BF$ on a $\Vir$-module $W$,
we say $W$ is of {\em central charge} $c$.

Following  \cite{DR1}, we define a notion of vertex operator algebra as follows:

\begin{definition}\label{def-voa}
A {\em vertex operator algebra} is a $\BZ$-graded vertex algebra
$V=\bigoplus_{n\in \BZ}V_{(n)}$ equipped with a vector $\omega\in V_{(2)}$, satisfying the conditions that
$\dim V_{(n)}<\infty$ for all $n\in \BZ$ and $V_{(n)}=0$ for $n$ sufficiently negative
and that
the components of the associated vertex operator written as $Y(\omega,x)=\sum_{n\in \BZ}L(n)x^{-n-2}$
provide a representation on $V$ of $\Vir$ of central charge $c\in \BF$ (an independent parameter) and
\begin{eqnarray}
&&L(0)u=nu \    \    \mbox{ for }u\in V_{(n)},\ n\in \BZ,\\
&&Y(L(-1)v,x)=\frac{d}{dx}Y(v,x)\    \    \mbox{ for }v\in V.
\end{eqnarray}
Note that the last condition amounts to $L(-1)=\D^{(1)}$ on $V$.
\end{definition}

Now, we come back to the Virasoro algebra $\Vir$. First, we have:

\begin{lemma}
The Virasoro algebra $\Vir$ is a $\B$-module Lie algebra with the action given by
\begin{equation}
    \D^{(k)}(L_m)=(-1)^k \binom{m+1}{k}L_{m-k},\   \   \   \    \D^{(k)}\bc=\delta_{k,0}\bc\  \  \
    \mbox{ for }k\in\BN, \  m\in\BZ.
\end{equation}
Furthermore, $U(\Vir)$ is a $\B$-module algebra.
\end{lemma}

\begin{proof} It is straightforward to show that $\Vir$ with the given action of $\B$
is a $\B$-module. Furthermore, for $k\in\BN$, $m,n\in\BZ$, using Lemma \ref{simple-facts} we have
\begin{align*}
    &\quad\ \D^{(k)}[L_{m},L_{n}]\\
    &=(-1)^k(m-n)\binom{m+n+1}{k}L_{m+n-k}+\frac{1}{2}\binom{m+1}{3}\delta_{m+n,0}\delta_{k,0}\bc\\
    &=\sum_{i=0}^k (-1)^k(m-n-k+2i)\binom{m+1}{k-i}\binom{n+1}{i}L_{m+n-k}+\frac{1}{2}\binom{m+1}{3}\delta_{m+n,0}\delta_{k,0}\bc\\
    &=\sum_{i=0}^k[\D^{(k-i)}L_m,\D^{(i)}L_n].
\end{align*}
This shows that $\Vir$ is a $\B$-module Lie algebra. From this the second assertion follows (cf. \cite[Lemma~3.2]{LM}).
\end{proof}

Set
\begin{equation}
    L(x)=\sum_{n\in\BZ}L_nx^{-n-2}\in \Vir[[x,x^{-1}]].
\end{equation}
Then the defining relations \eqref{Virasoro-relations} amount to
\begin{equation*}
    [L(x_1),L(x_2)]=(\partial_{x_2}^{(1)} L(x_2))x_2^{-1}\delta\left(\frac{x_1}{x_2}\right)-2L(x_2)\partial_{x_1}^{(1)} \left(x_2^{-1}\delta\left(\frac{x_1}{x_2}\right)\right)
    -\frac{1}{2}\partial_{x_1}^{(3)}\left(x_2^{-1}\delta\left(\frac{x_1}{x_2}\right)\right)\bc.
\end{equation*}
It follows that
\begin{equation}\label{eq:local-Vir}
    (x_1-x_2)^4[L(x_1),L(x_2)]=0.
\end{equation}

Equip the Virasoro algebra $\Vir$ with the following $\BZ$-grading
\begin{equation*}
    \Vir=\coprod_{n\in\BZ}\Vir_{(n)},
\end{equation*}
where
\begin{equation*}
    {\Vir}_{(0)}=\BF L_0\oplus \BF\bc\qquad \text{ and } \qquad {\Vir}_{(n)}=\BF L_{-n}\ \text{ for }n\ne0.
\end{equation*}
Set
\begin{equation*}
    {\Vir}_{-}=\coprod_{n\le -1} {\Vir}_{(n)}=\coprod_{n\ge 1} \BF L_{n} \qquad \text{ and }
    \qquad \CL_{\ge -1}=\coprod_{n\ge -1} \BF L_{n},
\end{equation*}
which are both Lie subalgebras.

Let $c\in\BF$. View $\BF$ as an $(\CL_{\ge -1}\oplus \BF \bc)$-module with
$\bc$ acting as scalar $c$ and with $\CL_{\ge -1}$ acting trivially.
Denote this $(\CL_{\ge -1}\oplus \BF \bc)$-module by $\BF_c$.
Then form an induced module
\begin{equation}
    V_{\Vir}(c,0)=U(\Vir)\otimes_{U(\CL_{\ge -1}+\BF \bc)}\BF_c.
\end{equation}
Set
\begin{equation}
    \1=1\otimes 1.
\end{equation}
It can be readily seen that $\CL_{\ge -1}+\BF (\bc-c)$ is a $\B$-submodule of $\Vir$,
so that $U(\Vir)(\CL_{\ge -1}+\BF (\bc-c))$ is a $\B$-submodule.
Since
\begin{equation*}
    V_{\Vir}(c,0)\cong U(\Vir)/U(\Vir)(\CL_{\ge -1}+\BF (\bc-c))
\end{equation*}
as a $U(\Vir)$-module, it follows that $V_{\Vir}(c,0)$ is a $(U(\Vir),\B)$-module
(cf.~\cite{LM}).
Furthermore, by \cite[Lemma~3.2]{LM}, we have
\begin{equation}\label{eq:conj-Vir}
    e^{z\D}L(x)e^{-z\D}=(e^{z\D}L(x))=L(x+z)
\end{equation}
on $V_{\Vir}(c,0)$.
With \eqref{eq:local-Vir} and \eqref{eq:conj-Vir},
by \cite[Theorem~2.14]{LM} we immediately have (cf.~\cite{DR2}):

\begin{proposition}
For any $c\in \BF$, there exists a vertex algebra structure on $V_{\Vir}(c,0)$,
which is uniquely determined by the condition that $\1$ is the vacuum vector and
\begin{equation*}
    Y(\omega,x)=L(x)\in(\End V_{\Vir}(c,0))[[x,x^{-1}]].
\end{equation*}
\end{proposition}

Just as in the case of characteristic zero, we have (cf.~\cite{LL}, \cite{DR2}):

\begin{proposition}
Every module for vertex algebra $V_{\Vir}(c,0)$ is naturally a locally truncated module
for the Virasoro algebra $\Vir$ of central charge $c$ with $L(x)=Y_W(\omega,x)$.
On the other hand, on every locally truncated module $W$ for $\Vir$  of central charge $c$,
there is a module structure $Y_{W}(\cdot,x)$ for vertex algebra $V_{\Vir}(c,0)$,
which is uniquely determined by the condition that
\begin{equation*}
 Y_W(\omega,x)=L(x)=\sum_{n\in \BZ}L(n)x^{-n-2}.
\end{equation*}
\end{proposition}

Note that the underlying space of vertex algebra $V_{\Vir}(c,0)$  is isomorphic to the $\Vir$-module
\begin{equation*}
    V_{\Vir}(c,0)=U(\Vir)/J_{c},
\end{equation*}
where $J_{c}=\sum_{n\ge -1}U(\Vir)L_{n}+U(\Vir)(\bc-c)$.
Recall that $\Vir$ is a $\BZ$-graded Lie algebra and $U(\Vir)$ is a $\BZ$-graded associative algebra with
\begin{equation*}
    \deg \bc =0,\     \    \     \    \deg L_{n}=-n\   \   \   \mbox{ for }n\in \BZ.
\end{equation*}
It is clear that $J_{c}$ is graded, so that $V_{\Vir}(c,0)$ is a $\BZ$-graded  $\Vir$-module where
\begin{equation}\label{grading-Virvoa}
    V_{\Vir}(c,0)_{(n)}=\spanf\{L(m_1)\ldots L(m_r)\1\mid m_i\in \BZ,\  m_1+\cdots+m_r=-n \}
\end{equation}
for $n\in \BZ$. In view of the P-B-W theorem, in fact we have
$V_{\Vir}(c,0)_{(n)}=0$ for $n<0$, $V_{\Vir}(c,0)_{(0)}=\BF \1$, $V_{\Vir}(c,0)_{(1)}=0$, and
\begin{equation*}
    V_{\Vir}(c,0)_{(n)}=\spanf\{L(-n_1)\ldots L(-n_r)\1\mid n_1\ge \cdots \ge n_r\ge 2,\  n_1+\cdots+n_r=n \}
\end{equation*}
for $n\ge 2$. Set
\begin{equation}
\omega=L(-2)\1\in V_{\Vir}(c,0)_{(2)}.
\end{equation}
As in the case of characteristic zero, we have (cf.~\cite{DR2}):

\begin{proposition}
For any $c\in \BF$, the vertex algebra $V_{\Vir}(c,0)$ equipped with  the $\BZ$-grading given in (\ref{grading-Virvoa})
and with $\omega=L(-2)\1$ as its conformal vector is a vertex operator algebra.
\end{proposition}

Next, we study quotient vertex algebras of the Virasoro vertex operator algebra
$V_{\Vir}(c,0)$.
First, we recall the definition of a restricted Lie algebra (see \cite{Jac}).

\begin{definition}
A {\em restricted Lie algebra} over $\BF$ (a field of characteristic $p$)  is a Lie algebra $\fg$ equipped with a mapping
$a\mapsto a^{[p]}$, called the {\em $p$-mapping},  satisfying the following conditions:
\begin{enumerate}
\item[(i)] $\ad a^{[p]}=(\ad a)^p$ for $a\in \fg$.
\item[(ii)] $(\lambda a)^{[p]}=\lambda^p a^{[p]}$ for  $\lambda\in\BF,\ a\in \fg$.
\item[(iii)] $(a+b)^{[p]}=a^{[p]}+b^{[p]}+\sum_{i=1}^{p-1}s_i(a,b)$  for $a,b\in \fg$,
\end{enumerate}
where $s_i(a,b)$ are give by
\begin{equation*}
    (\ad (a\otimes X+b\otimes 1))^{p-1}(a\otimes 1)=\sum_{i=1}^{p-1} is_i(a,b)\otimes X^{i-1}
\end{equation*}
in the Lie algebra $\fg\otimes \BF[X]$ over the polynomial ring $\BF[X]$.
\end{definition}

The following existence result can be found in \cite{Jac2} (cf. \cite[Theorem~2.3]{SF98}):

\begin{theorem}\label{p-mapping-sf}
Let $\fg$ be a Lie algebra with a basis $\{a_i\}_{i\in I}$.
Assume that  $(\ad a_i)^p=\ad b_i$ for $i\in I$, where $b_i\in \fg$.
Then there exists a restricted Lie algebra structure on $\fg$ with the $p$-mapping uniquely determined
by  $a_i^{[p]}=b_i$ for $i\in I$.
\end{theorem}

For any restricted Lie algebra $\fg$,  a basic fact is that
$a^{p}-a^{[p]}$  for $a\in \fg$ lie in the center of the universal enveloping algebra $U(\fg)$.
The {\em restricted universal enveloping algebra} $\mathfrak{u}(\fg)$ of a restricted Lie algebra $\fg$
by definition is the quotient algebra of $U(\fg)$ by the ideal generated by (central) elements $a^p-a^{[p]}$ for $a\in\fg$.

Using Theorem \ref{p-mapping-sf}, we have the following result (which might be known somewhere):

\begin{lemma}\label{th:restrictness}
There is a restricted Lie algebra structure on the Virasoro algebra $\Vir$, whose $p$-mapping is uniquely determined by
\begin{equation}\label{pmapping-Vir}
    \bc^{[p]}=\bc\  \   \mbox{ and }\  \    L_m^{[p]} = \begin{cases}
                        L_{p m} & p\mid m,\\
                        0 & p\nmid m
                    \end{cases}
\end{equation}
for $m\in\BZ$.
\end{lemma}

\begin{proof}
For $m,n\in\BZ$, we have
\begin{align*}
   &\phantom{=}\;\;(\ad L_m)^p L_n\\
    &=(\ad L_m)^{p-1} ((m-n)L_{m+n}+\frac{1}{2}\binom{m+1}{3}\delta_{m+n,0}\bc)\\
    &=(\ad L_m)^{p-2} ((m-n)(-n)L_{2m+n}+\frac{1}{2}\binom{m+1}{3}(m-n)\delta_{2m+n,0}\bc)\\
    &=\cdots\\
    &=\bigg(\prod_{i=0}^{p-1}(-n-(i-1)m)\bigg)L_{pm+n}
        +\frac{1}{2}\binom{m+1}{3}\delta_{pm+n,0}\bigg(\prod_{i=0}^{p-2}(-n-(i-1)m)\bigg)\bc.
        \end{align*}
Notice that we have $p\mid \prod_{i=0}^{p-2}(-n-(i-1)m)$  if $p\mid n$ (as $p\ge 3$)
and $\delta_{pm+n,0}=0$ if $p\nmid n$. Thus
\begin{equation}
 (\ad L_m)^p L_n=\bigg(\prod_{i=0}^{p-1}(-n-(i-1)m)\bigg)L_{pm+n}.
\end{equation}
Furthermore, if $p\nmid m$, we see $p\mid \prod_{i=0}^{p-1}(-n-(i-1)m)$. Then we obtain
\begin{equation}
    (\ad L_m)^p L_n=0\  \   \text{ if }p\nmid m.
\end{equation}

Assume $p\mid m$.
Noticing that $\prod_{i=0}^{p-1}(-n-(i-1)m)\equiv (-n)^p\equiv -n \pmod{p}$, we get
\begin{equation*}
    (\ad L_m)^p L_n=-n L_{pm+n}=(\ad L_{pm})L_n-\frac{1}{2}\binom{pm+1}{3}\delta_{pm+n,0}\bc.
\end{equation*}
If $p>3$, we see that $\binom{pm+1}{3}\equiv 0 \pmod{p}$.
If $p=3$, then $m=3k$ for some $k\in\BZ$, so that
$\binom{pm+1}{3}\equiv \binom{9k+1}{3}\equiv 0 \pmod{3}$.
Therefore, we have
\begin{equation}
    (\ad L_m)^p L_n=(\ad L_{pm})L_n \   \   \  \mbox{ if }p\mid m.
\end{equation}
It then follows from Theorem \ref{p-mapping-sf} that $\Vir$ is a restricted Lie algebra with
the $p$-mapping uniquely determined by (\ref{pmapping-Vir}).
\end{proof}

 For $m\in \BZ$, set
\begin{equation}
    \delta_{p|m}=\begin{cases}
      1& p\mid m,\\
      0& p\nmid m.
    \end{cases}
\end{equation}

As an immediate consequence of Lemma~\ref{th:restrictness}, we have:

\begin{corollary}
The elements $L_{n}^p-\delta_{p\mid n}L_{np}$  for $n\in\BZ$ lie
in the center of $U(\Vir)$.
\end{corollary}

Recall the following result  from  \cite{CMN}:

\begin{theorem}\label{th:CMN}
Let $\fg$ be a Lie algebra over $\BF$
and let $a_1,a_2,\ldots,a_p\in \fg$. Then
\begin{equation}
    \sum_{\sigma\in S_p}a_{\sigma(1)}a_{\sigma(2)}\ldots a_{\sigma(p)}=\sum_{\sigma\in S_p,\sigma(1)=1}[[\cdots[[a_{\sigma(1)},a_{\sigma(2)}],a_{\sigma(3)}],\ldots],a_{\sigma(p)}]
\end{equation}
in $U(\fg)$, where $S_p$ denotes the symmetric group.
\end{theorem}

For our need, we generalize this result slightly.
Let $1\le t\le p$ and let $r_1,\ldots,r_t$ be positive integers with $r_1+\cdots+r_t=p$.
Denote by $T(r_1,\ldots, r_t)$ the set of  maps
\begin{equation*}
    \tau:\  \{1,2,\dots,p\}\rightarrow \{1,2,\dots,t\}
\end{equation*}
such that  for each $1\le i\le t$, $\tau$ takes the value $i$ exactly $r_i$ times.

The following is a generalization of Theorem \ref{th:CMN}:

\begin{lemma}\label{th:CMN-r}
Let $\fg$ be a Lie algebra and let $a_1,\dots,a_t\in \fg$ with $1\le t\le p$.
Assume that $r_1,\ldots,r_t$ are positive integers such that $r_1+\cdots+r_t=p$.
Then
\begin{equation}\label{eq:T}
    r_1\sum_{\tau\in T(r_1, \ldots,r_t)}a_{\tau(1)}a_{\tau(2)}\ldots a_{\tau(p)}
    =\sum_{\tau\in T(r_1, \ldots, r_t),\tau(1)=1}[[\cdots[[a_{\tau(1)},a_{\tau(2)}],a_{\tau(3)}],\ldots],a_{\tau(p)}].
\end{equation}
\end{lemma}

\begin{proof} We define a sequence $b_1,b_2,\dots,b_p$ in $\fg$ by
\begin{align*}
    b_i&=a_1\quad\mbox{for }1\le i\le r_1,\\
    b_i&=a_2\quad\mbox{for }r_1+1\le i\le r_1+r_2,\\
    &\;\;\vdots\\
    b_i&=a_t\quad\mbox{for }p-r_t+1\le i\le p.
\end{align*}
By Lemma~\ref{th:CMN}, we have
\begin{equation}
    \sum_{\sigma\in S_p}b_{\sigma(1)}b_{\sigma(2)}\ldots b_{\sigma(p)}=\sum_{\sigma\in S_p,\sigma(1)=1}[[\cdots[[b_{\sigma(1)},b_{\sigma(2)}],b_{\sigma(3)}],\ldots],b_{\sigma(p)}],
\end{equation}
which gives
\begin{align*}
    &r_1!r_2!\cdots r_t!\sum_{\tau\in T(r_1, \ldots, r_t)}a_{\tau(1)}a_{\tau(2)}\ldots a_{\tau(p)}\\
    &=(r_1-1)!r_2!\cdots r_t!\sum_{\tau\in T(r_1, \ldots, r_t),\tau(1)=1}[[\cdots[[a_{\tau(1)},a_{\tau(2)}],a_{\tau(3)}],\ldots],a_{\tau(p)}].
\end{align*}
Since $r_1+\cdots+r_t=p$, we have $r_1\le p$ and $r_i<p$ for $2\le i\le t$, so that $p\nmid (r_1-1)!r_2!\cdots r_t!$.
Then \eqref{eq:T} follows immediately.
\end{proof}

The following is a key result in this section:

\begin{proposition}\label{Lnp-formula}
The following relation holds on $V_{\Vir}(c,0)$ for $n\ge 2$:
\begin{align*}
    Y((L_{-n}^p-\delta_{p\mid n}L_{-np})\1,x)&=\sum_{j\ge0}\binom{-n+1}{j}(-1)^j (L_{-n-j}^p-\delta_{p\mid(n+j)}L_{-(n+j)p})x^{jp}\\
    &\quad +\sum_{j\ge0}\binom{-n+1}{j}(-1)^{-n-j}(L_{j-1}^p-\delta_{p\mid(j-1)}L_{(j-1)p})x^{(-n+1-j)p}.
\end{align*}
\end{proposition}

\begin{proof} Let $n\ge 2$.
Since $L_{-n}^p-\delta_{p\mid n}L_{-np}$ is in the center of $U(\Vir)$,
we have
\begin{equation*}
    L_i(L_{-n}^p-\delta_{p\mid n}L_{-np})\1=(L_{-n}^p-\delta_{p\mid n}L_{-np})L_i\1=0
\end{equation*}
for all $i\ge-1$, as $L_i\1=0$.
Then, by Borcherds' commutator formula, we get
\begin{equation}\label{eq:LiDm=0}
    [L_s,((L_{-n}^p-\delta_{p\mid n}L_{-np})\1)_t]=\sum_{i\ge0}\binom{s+1}{i}(L_{i-1}(L_{-n}^p-\delta_{p\mid n}L_{-np})\1)_{s+t+1-i}=0
\end{equation}
on $V_{\Vir}(c,0)$ for $s,t\in \BZ$.

By Lemma~\ref{D-module}, we have
\begin{align*}
    Y(L_{-np}\1,x)&=Y(\D^{(np-2)}L_{-2}\1,x)=\partial_x^{(np-2)}Y(\omega,x)\\
    &=\partial_x^{(np-2)}\sum_{l\in\BZ}L_l x^{-l-2}\\
    &=\sum_{l\in\BZ}\binom{l+np-2}{np-2}L_{-np-l} x^{l}.
\end{align*}
On the other hand, by Lemma~\ref{th:a(x)^p}, we have
\begin{align*}
    Y(L_{-n}^p\1,x)=\Biggl(\sum_{j\ge0}\binom{-n+1}{j}(-1)^j L_{-n-j}x^{j}\Biggr)^p
            +\Biggl(\sum_{j\ge0}\binom{-n+1}{j}(-1)^{-n-j}L_{j-1}x^{-n+1-j}\Biggr)^p.
\end{align*}
Combining the two identities above we get
\begin{align}
    &\phantom{=}\;\;Y((L_{-n}^p-\delta_{p\mid n}L_{-np})\1,x)\nonumber\\
    &=\Biggl(\sum_{j\ge0}\binom{-n+1}{j}(-1)^j L_{-n-j}x^{j}\Biggr)^p
            +\Biggl(\sum_{j\ge0}\binom{-n+1}{j}(-1)^{-n-j}L_{j-1}x^{-n+1-j}\Biggr)^p\nonumber\\
    &\quad-\delta_{p\mid n}\sum_{l\in\BZ}\binom{l+np-2}{np-2}L_{-np-l} x^{l}.
\end{align}
To prove our assertion, we are going to determine the coefficient of $x^{m}$ in the expression above for every $m\in \BZ$.

Let $m\in \BN$. We see that the coefficient of $x^m$ in
$Y((L_{-n}^p-\delta_{p\mid n}L_{-np})\1,x)$ is equal to
\begin{align}\label{eq:L-n^p+m}
    \sum_{m_1+\cdots+m_p=m}\biggl(\prod_{i=1}^p \binom{-n+1}{m_i}\biggr)(-1)^m L_{-n-m_1}\ldots L_{-n-m_p}-\delta_{p\mid n}\binom{m+np-2}{np-2}L_{-np-m}.
\end{align}
First, consider the case $p\nmid m$. Let $m_1,\dots,m_p\in \BN$ with $m_1+\cdots +m_p=m$.
Suppose $n_1,\dots,n_t$ are all the distinct elements in the sequence $m_1,m_2,\dots, m_p$, and
for $1\le j\le t$, let $r_j$ be the number of occurrences of $n_j$.
Since $p\nmid m$, we have $t\ge 2$, which implies $r_1<p$. Noticing that
\begin{equation*}
    [[\cdots[[L_{-n-m_1},L_{-n-m_2}],L_{-n-m_3}],\ldots],L_{-n-m_p}]\in \BF L_{-np-m},
\end{equation*}
by Lemma~\ref{th:CMN-r} we see that the expression \eqref{eq:L-n^p+m} is equal to $\alpha L_{-np-m}$ for some $\alpha \in\BF$.
That is, the coefficient of $x^m$ in
$Y((L_{-n}^p-\delta_{p\mid n}L_{-np})\1,x)$ equals $\alpha L_{-np-m}$ for some $\alpha \in\BF$.
By \eqref{eq:LiDm=0}, we have
$[\alpha L_{-np-m}, L_{s}]=0$ on $V_{\Vir}(c,0)$ for all $s\in \BZ$, which gives
\begin{equation*}
    (-np-m-s)\alpha L_{-np-m+s}\1+\frac{1}{2} \binom{-np-m+1}{3}\alpha\delta_{np+m,s}\ell\1=0.
\end{equation*}
Let $s\in \BZ$ be such that $-np-m-s\equiv 1\ \pmod{p}$ and $s-np-m\le -2$.
Then we get $\alpha L_{-np-m+s}\1=0$, noticing that $\delta_{np+m,s}=0$.
Since $L_{-np-m+s}\1\ne 0$ in $V_{\Vir}(c,0)$, we must have $\alpha =0$.
Thus, the coefficient of $x^{m}$ is zero in this case.

Assume $p\mid m$. Write $m=kp$ with $k\in \BN$. In this case, by a similar argument we see
that \eqref{eq:L-n^p+m} is equal to
\begin{align*}
    \binom{-n+1}{k}^p(-1)^{kp}L_{-n-k}^p+ \beta L_{-np-m}
\end{align*}
for some $\beta\in\BF$. By \eqref{eq:LiDm=0}, we have
\begin{equation*}
    \left[L_{-1},\binom{-n+1}{k}^p(-1)^{kp}L_{-n-k}^p+ \beta L_{-np-m}\right]=0.
\end{equation*}
As $L_{-n-k}^p-\delta_{p\mid(n+k)}L_{-np-kp}$, which lies in the center of $U(\Vir)$, commutes with $L_{-1}$, we get
\begin{equation*}
    \left(\beta+\delta_{p\mid(n+k)}(-1)^{kp}\binom{-n+1}{k}^p\right)[ L_{-1},L_{-np-m}]=0
\end{equation*}
on $V_{\rm Vir}(c,0)$. Since
\begin{equation*}
    [ L_{-1}, L_{-np-m}]\1=(np+m-1)L_{-np-m-1}\1=-L_{-np-m-1}\1\ne 0\ \mbox{ in }V_{\Vir}(c,0),
\end{equation*}
noticing that $\delta_{-np-m-1,0}=0$ and $-np-m-1\le -2$,
we obtain
\begin{equation*}
    \beta= -\delta_{p\mid(n+k)}(-1)^{kp}\binom{-n+1}{k}^{p}.
\end{equation*}
Thus, by Fermat's little theorem,
we see that the coefficient of $x^{m}$ is
\begin{align*}
    &\phantom{=}\;\;\binom{-n+1}{k}^p(-1)^{kp}(L_{-n-k}^p-\delta_{p\mid(n+k)} L_{-np-kp})\\
    &=\binom{-n+1}{k} (-1)^k(L_{-n-k}^p-\delta_{p\mid(n+k)} L_{-np-kp}).
\end{align*}

Next, we consider the coefficient of $x^{-m}$ with $m\ge (n-1)p$ in $Y((L_{-n}^p-\delta_{p\mid n}L_{-np})\1,x)$,
which is
\begin{align}\label{eq:L-n^p-m}
    \sum_{m_1+\cdots+m_p=m-(n-1)p}\biggl(\prod_{i=1}^p \binom{-n+1}{m_i}\biggr)(-1)^{-m-p} L_{m_1-1}\ldots L_{m_p-1}-\delta_{p\mid n}\binom{np-m-2}{np-2}L_{-np+m}.
\end{align}
Assume $p\nmid m$. By a similar argument used above, we see that \eqref{eq:L-n^p-m} is equal to $\alpha L_{-np+m}$
for some $\alpha \in\BF$. Then by \eqref{eq:LiDm=0}, we have
$[\alpha L_{-np+m}, L_{s}]=0$ on $V_{\Vir}(c,0)$ for all $s\in \BZ$, which implies
\begin{equation*}
    (-np+m-s)\alpha L_{-np+m+s}\1+\frac{1}{2} \binom{-np+m+1}{3}\alpha\delta_{-np+m+s,0}\ell\1=0.
\end{equation*}
Let $s\in\BZ$ be such that $-np+m-s\equiv 1 \pmod{p}$ and $-np+m+s\le -2$.
Then we have $\alpha L_{-np+m+s}\1=0$, noticing that $\delta_{-np+m+s,0}=0$.
Since $L_{-np+m+s}\1\ne 0$ in $V_{\Vir}(c,0)$,
it follows that $\alpha=0$, which shows that the coefficient of $x^{-m}$ is zero.

Assume $p\mid m$. Write $m=kp$ for some $k$. Similarly, we obtain \eqref{eq:L-n^p-m} is equal to
\begin{align*}
    \binom{-n+1}{k}^p (-1)^{-kp-p} L_{k-n}^p+\beta L_{kp-np}
\end{align*}
for some $\beta\in\BF$. Then by \eqref{eq:LiDm=0}, we have
\begin{equation*}
    \left[L_s,\binom{-n+1}{k}^p (-1)^{-kp-p} L_{k-n}^p+\beta L_{kp-np}\right]=0
\end{equation*}
for every $s\in\BZ$.
Since $L_{k-n}^p-\delta_{p\mid(k-n)}L_{kp-np}$ lies in the center of $U(\Vir)$,
it follows that
\begin{equation*}
    \left(\beta+\delta_{p\mid(k-n)}(-1)^{-kp-p}\binom{-n+1}{k}^p\right)[L_s,L_{kp-np}]=0
\end{equation*}
on $V_{\Vir}(c,0)$.
Let $s\in\BZ$ be such that $s\equiv 1 \pmod{p}$ and $s+kp-np\le -2$.
Then $[L_s,L_{kp-np}]=(s-kp+np)L_{s+kp-np}=L_{s+kp-np}$, and $L_{s+kp-np}\1\ne0$.
It follows that
\begin{equation*}
    \beta=-\delta_{p\mid(k-n)}(-1)^{-kp-p}\binom{-n+1}{k}^p.
\end{equation*}
By Fermat's little theorem,
we see that the the coefficient of $x^{-m}$ is
\begin{align*}
    &\phantom{=}\;\;\binom{-n+1}{k}^p (-1)^{-kp-p} (L_{k-n}^p-\delta_{p\mid(k-n)} L_{kp-np})\\
    &=\binom{-n+1}{k} (-1)^{-k-1} (L_{k-n}^p-\delta_{p\mid(k-n)} L_{kp-np}).
\end{align*}
This completes the proof.
\end{proof}

As an immediate consequence of Proposition \ref{Lnp-formula}, we have:

\begin{corollary}\label{some-facts}
The following hold in $V_{\Vir}(c,0)$ for $m\in \BN$ and $n\ge 2$:
\begin{equation}
    \D^{(m)}(L_{-n}^p-\delta_{p\mid n}L_{-np})\1=\begin{cases}
        (-1)^k\binom{-n+1}{k}(L_{-n-k}^p-\delta_{p\mid(n+k)}L_{-(n+k)p})\1& \text{ if } m=kp,\\
                                                                      0 & \text{ otherwise}.
                                               \end{cases}
\end{equation}
Furthermore, the subspace spanned by $(L_{-n}^p-\delta_{p\mid n}L_{-np})\1$ for $n\ge2$ is a $\B$-submodule.
\end{corollary}

Recall  that a left ideal of a vertex algebra $V$ is an ideal if and only if it is a $\B$-submodule of $V$.
As for vertex algebra $V_{\Vir}(c,0)$, we see that a left ideal amounts to a $\Vir$-submodule of $V_{\Vir}(c,0)$.
Using Corollary \ref{some-facts}, we immediately obtain:

\begin{proposition}\label{th:ideal-Vir}
For any $\mu\in \BF$, the $\Vir$-submodule of $V_{\Vir}(c,0)$ generated by
\begin{equation*}
    (L_{-n}^p-\delta_{p\mid n}L_{-np}-\delta_{n,2}\mu^p)\1\   \   \   \mbox{  for }n\ge 2
\end{equation*}
is an ideal of vertex algebra $V_{\Vir}(c,0)$, which is denoted by $I_{\mu}$.
Furthermore, $I_{\mu}$ equals the ideal of $V_{\Vir}(c,0)$ generated by $(L_{-2}^p-\mu^{p})\1$.
\end{proposition}

For $\mu\in \BF$, set
\begin{equation}\label{Vlambda}
V^{\mu}_{\Vir}(c,0)=V_{\Vir}(c,0)/I_{\mu},
\end{equation}
a $\Vir$-module and a vertex algebra. Note that the ideal $I_{\mu}$ is graded  if and only if $\mu=0$.
Especially, $V^{0}_{\Vir}(c,0)$ is a $\BZ$-graded vertex operator algebra.

\section{The Zhu algebras $A(V_{\Vir}(c,0))$ and $A(V^{0}_{\Vir}(c,0))$}
In this section, we continue to study the quotient vertex operator algebra $V^{0}_{\Vir}(c,0)$ of  $V_{\Vir}(c,0)$
for a general $c\in \BF$.
More specifically, we shall classify irreducible $\BN$-graded  $V^{0}_{\Vir}(c,0)$-modules
by determining the Zhu algebra $A(V^{0}_{\Vir}(c,0))$ explicitly.

Recall the vertex operator algebra $V_{\Vir}(c,0)$.
The following follows from the same arguments as in the case of characteristic zero
(see  \cite{W}; cf. \cite{DN}):

\begin{lemma}\label{th:L-nW}
For $n\ge 1$ and $v\in V_{\Vir}(c,0)$, we have
\begin{equation*}
    L_{-n}v\equiv (-1)^n\big((n-1)(L_{-2}+L_{-1})+L_0\big)v \pmod{O(V_{\Vir}(c,0))}.
\end{equation*}
\end{lemma}

Recall that for $u\in V_{\Vir}(c,0)$,
\begin{equation*}
    [u]=u+O(V_{\Vir}(c,0)) \in A(V_{\Vir}(c,0)).
\end{equation*}
From \cite{DR1}, we have
\begin{equation}\label{fromDR}
[(L_{-2}+L_{-1})u]=[\omega]*[u]\    \    \    \mbox{  for }u\in V_{\Vir}(c,0),
\end{equation}
where $\omega=L_{-2}\1$ is the conformal vector.
Using the same arguments as in the case of characteristic zero
(see \cite{FZ}), we have:

\begin{proposition}\label{th:AV-Vir}
The linear map $f(x)\mapsto f([\omega])$ for $f(x)\in \BF[x]$ is an algebra isomorphism from
 $\BF[x]$ onto $A(V_{\Vir}(c,0))$.
\end{proposition}

Furthermore, we have:

\begin{lemma}\label{th:L-n^pomega}
The following relations hold in $A(V_{\Vir}(c,0))$ for $n\ge 1$:
 \begin{equation}\label{pcenter-Zhu}
 [(L_{-n}^p -\delta_{p\mid n}L_{-np})\1]= (-1)^{pn}(n-1)([\omega]^p-[\omega]).
 \end{equation}
In particular, we have $[L_{-2}^p \1] =[\omega]^p-[\omega]$.
\end{lemma}

\begin{proof}
Using Lemma~\ref{th:L-nW} and equation (\ref{fromDR}), we have
\begin{align*}
    [L_{-n}^p\1]
    &= (-1)^n\big[\big((n-1)(L_{-2}+L_{-1})+L_0\big)L_{-n}^{p-1}\1\big]\\
    &= (-1)^n\big((n-1)[\omega]+n(p-1)\big)*[L_{-n}^{p-1}\1]\\
    &= (-1)^{2n}\big((n-1)[\omega]+n(p-1)\big)*\big((n-1)[\omega]+n(p-2)\big)*[L_{-n}^{p-2}\1]\\
    &=\cdots\\
    &= (-1)^{pn}\big((n-1)[\omega]+n(p-1)\big)*\big((n-1)[\omega]+n(p-2)\big)*\cdots*(n-1)[\omega].
\end{align*}
If $p\nmid n$, noticing that $\{ kn+p\BZ\mid 0\le k\le p-1\}=\BZ_{p}\subset \BF$,  we have
\begin{align*}
    &\phantom{=}\;\;(-1)^{pn}\big((n-1)[\omega]+n(p-1)\big)*\big((n-1)[\omega]+n(p-2)\big)*\cdots*(n-1)[\omega]\\
    &= (-1)^{pn} \big((n-1)^p[\omega]^p- (n-1)[\omega]\big)\\
    &= (-1)^{pn}(n-1) \big([\omega]^p-[\omega]\big).
\end{align*}
If $p\mid n$, we have
\begin{align*}
    &\phantom{=}\;\;(-1)^{pn}\big((n-1)[\omega]+n(p-1)\big)*\big((n-1)[\omega]+n(p-2)\big)*\cdots*(n-1)[\omega]\\
    &= (-1)^{pn}(n-1)[\omega]^p.
\end{align*}
Thus
\begin{eqnarray}
 [L_{-n}^p\1]=(-1)^{pn}(n-1) \big([\omega]^p-[\omega]\big)+\delta_{p\mid n}(-1)^{pn}(n-1)[\omega].
\end{eqnarray}
On the other hand, by Lemma \ref{th:L-nW} we have
\begin{equation*}
    [L_{-np}\1]=[(-1)^{pn}\big((n-1)(L_{-2}+L_{-1})+L_0\big)\1]=(-1)^{pn}(n-1)[\omega].
\end{equation*}
Then (\ref{pcenter-Zhu}) follows immediately.
\end{proof}

Recall that $V^{0}_{\Vir}(c,0)=V_{\Vir}(c,0)/I_0$, where
$I_0$ is an ideal which is the $\Vir$-submodule of $V_{\Vir}(c,0)$
generated by $(L_{-n}^p -\delta_{p\mid n}L_{-np})\1$ for $n\ge 2$ (see Proposition~\ref{th:ideal-Vir}).

As the main result of this section, we have:

\begin{theorem}\label{main-Vir}
For any $c\in \BF$, the algebra isomorphism in Proposition \ref{th:AV-Vir} from $\BF[x]$ to $A(V_{\Vir}(c,0))$
reduces to an isomorphism from  $\BF[x]/(x^p-x)\BF[x]$ to $A(V^{0}_{\Vir}(c,0))$.
\end{theorem}

\begin{proof} Set $\bar{I}_{0}=\{[u]\mid u\in I_0\}\subset A(V_{\Vir}(c,0))$.  From \cite{FZ},
$\bar{I}_{0}$ is an ideal and
$A(V^{0}_{\Vir}(c,0))\cong A(V_{\Vir}(c,0))/\bar{I}_{0}$.
Now, it suffices to prove
\begin{equation*}
    \{[u]\mid u\in I_0\}\  (=\bar{I}_{0})=([\omega]^p-[\omega])\BF[\omega].
\end{equation*}
By Lemma~\ref{th:L-n^pomega},
we have $[(L_{-n}^p -\delta_{p\mid n}L_{-np})\1] \in \BF\cdot([\omega]^p-[\omega])$
for all $n\ge 2$. Let $u$ be any $\BZ$-homogeneous vector in $V_{\Vir}(c,0)$ such that $[u]\in ([\omega]^p-[\omega])\BF[\omega]$.
 For any positive integer $m$,  using Lemma~\ref{th:L-nW} and equation (\ref{fromDR}),  we get
\begin{align*}
    [L_{-m}u]= (-1)^m[\big((m-1)(L_{-2}+L_{-1})+L_0\big)u]
    = (-1)^m(m-1)[\omega][u]+(-1)^m (\deg u)[u],
\end{align*}
which lies in $([\omega]^p-[\omega])\BF[\omega]$.
As $I_0$ is a $\Vir$-submodule of $V_{\Vir}(c,0)$
generated by $(L_{-n}^p -\delta_{p\mid n}L_{-np})\1$ for $n\ge 2$,
it then follows  that $\{[u]\mid u\in I_0\}\subset ([\omega]^p-[\omega])\BF[\omega]$.

On the other hand, using Lemma~\ref{th:L-n^pomega} and equation (\ref{fromDR}) we get
\begin{equation*}
    [(L_{-2}+L_{-1})^k L_{-2}^p\1]=[\omega]^k ([\omega]^p-[\omega])
\end{equation*}
for $k\ge 0$. Notice that $(L_{-2}+L_{-1})^k L_{-2}^p\1 \in I_0$ as $L_{-2}^{p}\1\in I_0$.
It then follows that $([\omega]^p-[\omega])\BF[\omega]\subset\{[u]\mid u\in I_0\}$.
Therefore, we have $\{[u]\mid u\in I_0\}=([\omega]^p-[\omega])\BF[\omega]$. This completes the proof.
\end{proof}

As $\Vir$ is a $\BZ$-graded Lie algebra, $U(\Vir)$ is a $\BZ$-graded associative algebra.
Let $c,\lambda\in \BF$. Set
\begin{equation}
J(c,\lambda)=U(\Vir)(\bc-c)+U(\Vir)(L_0-\lambda)+\sum_{m\ge 1}U(\Vir)L_{m},
\end{equation}
a graded left ideal of $U(\Vir)$. Then define a Verma module
\begin{equation}
M_{\Vir}(c,\lambda)=U(\Vir)/J(c,\lambda),
\end{equation}
which is naturally a $\BZ$-graded $\Vir$-module.
Set
\begin{equation*}
    v_{c,\lambda}=1+J(c,\lambda)\in M_{\Vir}(c,\lambda),
\end{equation*}
a generator of $\Vir$-module $M_{\Vir}(c,\lambda)$.
We have
\begin{equation*}
    M_{\Vir}(c,\lambda)=\bigoplus_{n\in \BN}M_{\Vir}(c,\lambda)(n),
\end{equation*}
where $M_{\Vir}(c,\lambda)(0)=\BF v_{c,\lambda}$.
 Denote by $L_{\Vir}(c,\lambda)$ the quotient module of $M_{\Vir}(c,\lambda)$ by the (unique) maximal graded
 submodule. Then $L_{\Vir}(c,\lambda)$ is an irreducible $\BN$-graded $\Vir$-module.
On the other hand, it is clear that  any irreducible $\BN$-graded $\Vir$-module is of this form.
 In terms of $V_{\Vir}(c,0)$-modules, we have that  $L_{\Vir}(c,\lambda)$ is
 an irreducible $\BN$-graded $V_{\Vir}(c,0)$-module for any $\lambda\in \BF$ and
 any irreducible $\BN$-graded $V_{\Vir}(c,0)$-module is of this form.

 Combining Propositions \ref{WtoA(V)module} and \ref{A(V)toV} with
 Theorem \ref{main-Vir} we immediately have:

\begin{corollary}
For every $n\in\BF_p=\{0,1,\ldots,p-1\}$,
$L_{\Vir}(c,n)$ is an irreducible $\BN$-graded $V_{\Vir}^{0}(c,0)$-module.
Furthermore, $L_{\Vir}(c,n)$ for $n\in\BF_p$ form
a complete set of equivalence class representatives of irreducible $\BN$-graded  $V_{\Vir}^{0}(c,0)$-modules.
\end{corollary}

On the other hand, we have:

\begin{proposition}
For any $c\in \BF$, vertex algebra $V_{\Vir}^{0}(c,0)$ is $C_{2}$-cofinite.
\end{proposition}

\begin{proof} Notice that $\omega_{n}=L(n-1)$ for $n\in \BZ$. In view of the P-B-W theorem, we have
\begin{equation*}
    V_{\Vir}^{0}(c,0)={\rm span}\{ \omega_{-n_1}\cdots \omega_{-n_r}\1\mid  r\ge 0,\ n_i\ge 1\}.
\end{equation*}
From Lemma \ref{symmetric-algebra},  the commutative associative algebra $V_{\Vir}^{0}(c,0)/C_{2}(V_{\Vir}^{0}(c,0))$
is generated by $\bar{\omega}$.
As $L(-2)^{p}\1=0$ in $V_{\Vir}^{0}(c,0)$, we have $\omega_{-1}^{p}\1=0$. Thus $\bar{\omega}^{p}=0$ in
$V_{\Vir}^{0}(c,0)/C_2(V_{\Vir}^{0}(c,0) )$. It follows immediately that it is finite dimensional.
\end{proof}

\section{Affine vertex algebras $V_{\hat\fg}(\ell,0)$ and their quotients $V^{\chi}_{\hat\fg}(\ell,0)$}
In this section, we study affine vertex algebras, which are vertex algebras associated to affine Lie algebras.
In particular, we study the quotients $V^{\chi}_{\hat\fg}(\ell,0)$ of
the universal affine vertex algebra $V_{\hat\fg}(\ell,0)$ by ideals associated to the $p$-center of
the affine Lie algebra $\hat{\fg}$.

Let $\fg$ be a Lie algebra over $\BF$ with a
symmetric invariant bilinear form $\langle\cdot,\cdot\rangle$. Recall the associated affine Lie algebra
$\hat{\fg}$, where
\begin{equation*}
    \hat{\fg}=\fg\otimes \BF[t,t^{-1}] +\BF \bk
\end{equation*}
as a vector space and the Lie bracket is given by $[\hat{\fg},\bk]=0$ and
\begin{equation*}
    [a\otimes t^{m},b\otimes t^{n}]=[a,b]\otimes t^{m+n}+\langle a,b\rangle \delta_{m+n,0}\bk
\end{equation*}
for $a,b\in \fg,\  m,n\in \BZ$.

As a tradition, for $a\in \fg,\ n\in \BZ$, let $a(n)$ denote the corresponding operator
on $\hat{\fg}$-modules. In practice, $a\otimes t^{n}$ is also denoted by $a(n)$.
For $a\in \fg$, set
\begin{equation*}
    a(x)=\sum_{n\in \BZ}a(n)x^{-n-1}\in \hat{\fg}[[x,x^{-1}]].
\end{equation*}
We also consider $a(x)$ as an element of $(\End W)[[x,x^{-1}]]$ for any $\hat{\fg}$-module $W$.
We call a $\hat\fg$-module $W$ a {\em locally truncated module} if for every $w\in W$ and $a\in\fg$,
$a(n)w=0$ for $n$ sufficiently large.
On the other hand, if $\bk$ acts as a scalar $\ell\in\BF$ on a $\hat\fg$-module $W$,
we say $W$ is of {\em level} $\ell$.

Let $\ell\in \BF$. Denote by $\BF_{\ell}$ the $1$-dimensional $(\fg[t]\oplus \BF \bk)$-module $\BF$, on which
$\fg[t]$ acts trivially and $\bk$ acts as scalar $\ell$. Then form an induced
 $\hat{\fg}$-module
\begin{equation}
V_{\hat\fg}(\ell,0)=U(\hat{\fg})\otimes_{U(\fg[t]+\BF \bk)} \BF_{\ell},
\end{equation}
which is commonly referred to as the {\em level $\ell$ universal vacuum module} of $\hat\fg$.

Set ${\bf 1}=1\otimes 1\in V_{\hat\fg}(\ell,0)$.
Identify $\fg$ as a subspace of $V_{\hat\fg}(\ell,0)$ through the linear map
\begin{equation*}
    \fg\rightarrow V_{\hat\fg}(\ell,0); \  \   a\mapsto a(-1){\bf 1}.
\end{equation*}

Just as in the case of characteristic $0$, we have (a detailed proof can be found in \cite{LM}):

\begin{proposition}\label{th:affineVA}
For any $\ell\in \BF$, there exists a vertex algebra structure
on $V_{\hat\fg}(\ell,0)$, which is uniquely determined by the condition that $\1=1\otimes 1$ is the vacuum vector and
\begin{equation}
    Y(a,x)=a(x)\in(\End V_{\hat\fg}(\ell,0))[[x,x^{-1}]]\quad \text{for }a\in\fg.
\end{equation}
\end{proposition}

Note that for $a\in \fg\subset V_{\hat\fg}(\ell,0),\ n\in \BZ$, we have
\begin{equation}
a_{n}=a(n)\   \   \  \mbox{ on } V_{\hat\fg}(\ell,0).
\end{equation}
It follows that $\fg$ generates $V_{\hat{\fg}}(\ell,0)$ as a vertex algebra and a left ideal of $V_{\hat{\fg}}(\ell,0)$
exactly amounts to a $\hat{\fg}$-submodule.

We make $\hat\fg$ a $\BZ$-graded Lie algebra by defining
\begin{equation*}
    \deg \bk=0,\   \   \   \   \deg (\fg\otimes t^{n})=-n\   \   \mbox{ for }n\in \BZ.
\end{equation*}
Similar to the Virasoro vertex algebra $V_{\Vir}(c,0)$,
$V_{\hat{\fg}}(\ell,0)$ is naturally a $\BZ$-graded $\hat\fg$-module. We fix this $\BZ$-grading of $V_{\hat{\fg}}(\ell,0)$:
\begin{equation}\label{Zgrading-affine}
V_{\hat{\fg}}(\ell,0)=\bigoplus_{n\in \BZ}V_{\hat{\fg}}(\ell,0)_{(n)},
\end{equation}
where
$V_{\hat{\fg}}(\ell,0)_{(n)}=0$ for $n<0$,  $V_{\hat{\fg}}(\ell,0)_{(0)}=\BF \1$, and
$V_{\hat{\fg}}(\ell,0)_{(1)}=\fg$. For $a\in \fg$, as
\begin{equation*}
    Y(a,x)=\sum_{n\in \BZ}a_{n}x^{-n-1}\   \   \mbox{with }a_{n}=a(n)
\end{equation*}
we have
\begin{equation}
a_{m}V_{\hat{\fg}}(\ell,0)_{(n)}\subset V_{\hat{\fg}}(\ell,0)_{(n-m)} \    \   \mbox{ for }m,n\in \BZ.
\end{equation}
As $\fg$ generates $V_{\hat{\fg}}(\ell,0)$ as a vertex algebra,
using induction we readily have:

\begin{proposition}
For any $\ell\in \BF$, vertex algebra $V_{\hat{\fg}}(\ell,0)$ equipped with the given $\BZ$-grading
is a $\BZ$-graded vertex algebra.
\end{proposition}

Recall the following result from \cite{LM}:

\begin{lemma}\label{th:Dkp}
For any $a\in\fg$, $k\in\BN$ and $n\in\BZ_+$, the following hold in $V_{\hat\fg}(\ell,0)$:
\begin{align*}
    \D^{(kp)}(a(-n)^p\1)&=\binom{n+k-1}{k}a(-n-k)^p\1,\\
    \D^{(k)}(a(-n)^p\1)&=0\text{ if }p\nmid k.
\end{align*}
\end{lemma}

We have the following technical result on binomial coefficients:

\begin{lemma}\label{Tresults}
\begin{enumerate}
\item[{\rm(i)}] $\binom{pn+k-1}{k}\equiv 0 \pmod{p}$ for $n,k\in\BZ_+$ with $p\nmid k$.
\item[{\rm(ii)}] $\binom{pn+pk-1}{pk}\equiv \binom{n+k-1}{k} \pmod{p}$ for $n,k\in\BZ_+$.
\end{enumerate}
\end{lemma}

\begin{proof}
For each nonnegative integer $m$, write $m=\sum_{i\ge0}m_ip^i$,
where $m_i\in\{0,1,\ldots,p-1\}$.
For (i), as $p\nmid k$,
we have $(pn+k-1)_0=k_0-1< k_0$, so that
$\binom{(pn+k-1)_0}{k_0}=0$. Then
$\binom{pn+k-1}{k}\equiv 0 \pmod{p}$ by Lucas' theorem.
For (ii), note that $pm-1=p(m-1)+p-1$, which implies  $(pm-1)_{0}=p-1$
and $(pm-1)_{i}=(m-1)_{i-1}$ for $i\ge 1$.
Then, using Lucas' theorem again we get
\begin{align*}
\binom{pn+pk-1}{pk}&\equiv
\prod_{i\ge0}\binom{(pn+pk-1)_i}{(pk)_i}
\equiv \binom{p-1}{0}\prod_{i\ge1}\binom{(pn+pk-1)_i}{(pk)_i}\\
&\equiv \prod_{i\ge0}\binom{(n+k-1)_i}{k_i}
\equiv \binom{n+k-1}{k} \pmod{p},
\end{align*}
as desired.
\end{proof}

The following result can be found in  \cite{Ma}:

\begin{proposition}
Let $\fg$ be a restricted Lie algebra equipped with a symmetric invariant bilinear form $\langle\cdot,\cdot\rangle$.
Then there is a restricted Lie algebra structure on $\hat\fg$,
where the $p$-mapping is given by $a(n)^{[p]}=a^{[p]}(pn)$ for $a\in \fg,\ n\in\BZ$ and  by $\bk^{[p]}=\bk$,
where  $a\mapsto a^{[p]}$ denotes the $p$-mapping of $\fg$.
\end{proposition}

For the rest of this section, we assume that {\em $\fg$ is a restricted Lie algebra} (with a fixed $p$-mapping).
Let $\chi\in \fg^*$, a linear functional. Denote by $J_\chi$ the $\hat\fg$-submodule of $V_{\hat\fg}(\ell,0)$
generated by vectors
\begin{equation}
    \big(a(-m)^p-a^{[p]}(-pm)-\delta_{m,1}\chi(a)^p\big)\1
\end{equation}
for $a\in \fg,\ m\in\BZ_+$. We have:

\begin{proposition}\label{Jchi}
The $\hat\fg$-submodule $J_\chi$ is an ideal of vertex algebra $V_{\hat\fg}(\ell,0)$, which equals the ideal
generated by $\big(a(-1)^p-a^{[p]}(-p)-\chi(a)^p\big)\1$ for $a\in \fg$.
\end{proposition}

\begin{proof}
As $\fg$ generates $V_{\hat\fg}(\ell,0)$ as a vertex algebra, it follows that
$J_\chi$ is a left ideal of $V_{\hat\fg}(\ell,0)$.
Note that $\D^{(k)}\1=\delta_{k,0}\1$ for $k\in\BN$.
By Lemmas~\ref{th:Dkp} and \ref{Tresults}, for $k,n\in\BZ_+$, we have
\begin{align*}
    \D^{(k)}(a(-n)^p\1-a^{[p]}(-pn)\1)&=0 \text{ if } p\nmid k,\\
    \D^{(pk)}(a(-n)^p\1-a^{[p]}(-pn)\1)&=\binom{n+k-1}{k}a(-n-k)^p\1-\binom{pn+pk-1}{pk}a^{[p]}(-pn-pk)\1\\
    &=\binom{n+k-1}{k}(a(-n-k)^p-a^{[p]}(-pn-pk))\1.
\end{align*}
It then follows that $J_\chi$ is a $\B$-submodule of $V_{\hat\fg}(\ell,0)$.
Therefore, $J_\chi$ is an ideal of $V_{\hat\fg}(\ell,0)$.
Furthermore, notice that for $a\in \fg,\ k\in \BZ_{+}$, we have
\begin{equation*}
    \D^{(pk)}(a(-1)^p\1-a^{[p]}(-p)\1-\chi(a)^p \1)=(a(-1-k)^p-a^{[p]}(-p-pk))\1.
\end{equation*}
It then follows that $J_{\chi}$ is the ideal generated by
$\big(a(-1)^p-a^{[p]}(-p)-\chi(a)^p\big)\1$ for $a\in \fg$.
\end{proof}

For $\chi\in \fg^*$, set
\begin{equation}
    V^\chi_{\hat\fg}(\ell,0)=V_{\hat\fg}(\ell,0)/J_\chi,
\end{equation}
a $\hat{\fg}$-module of level $\ell$ and a vertex algebra.
We have the following characterization of $V^{\chi}_{\hat\fg}(\ell,0)$-modules:

\begin{proposition}\label{th:L-module-property}
If $W$ is a $V^{\chi}_{\hat\fg}(\ell,0)$-module, then
$W$ is a locally truncated $\hat\fg$-module of level $\ell$ such that
 $a(n)^p-a^{[p]}(np)-\delta_{n,-1}\chi(a)^p=0$ on $W$ for $a\in \fg,\  n\in\BZ$.
Furthermore, the converse is also true.
\end{proposition}

\begin{proof} Note that a $V^{\chi}_{\hat\fg}(\ell,0)$-module is naturally a $V_{\hat\fg}(\ell,0)$-module.
Recall that a $V_{\hat\fg}(\ell,0)$-module exactly amounts to a locally truncated
$\hat\fg$-module of level $\ell$. Let $W$ be a $V_{\hat\fg}(\ell,0)$-module.
Let $a\in \fg$. Noticing that $a_ia=\delta_{i,1}\langle a,a\rangle \1$ for $i\in \BN$,
by Corollary \ref{th:a_n^p}
we have
\begin{equation}\label{eq:L-module-property02}
    Y_W(a(-1)^p\1,x)
    =\sum_{j\in\BZ} a(j)^p x^{(-j-1)p}.
\end{equation}
On the other hand,
by Lemma~\ref{D-module}, we have
\begin{align*}
Y_W(a^{[p]}(-p)\1,x)
    &=Y_{W}(\D^{(p-1)}a^{[p]},x)
    =\partial^{(p-1)}_{x}Y_{W}(a^{[p]},x)=\sum_{n\in \BZ}\binom{-n-1}{p-1} a^{[p]}(n) x^{-n-p}\\
&=\sum_{n\in \BZ}(-1)^{p-1}\binom{n+p-1}{p-1} a^{[p]}(n) x^{-n-p}.
\end{align*}
Notice that $(-1)^{p-1}=1$ and that $\binom{n+p-1}{p-1}\equiv 0 \pmod{p}$ if $p\nmid n$
and $\binom{n+p-1}{p-1}\equiv 1 \pmod{p}$ if $p\mid n$.
Then
\begin{equation}\label{eq:L-module-property01}
\begin{split}
    Y_W(a^{[p]}(-p)\1,x)
     =\sum_{j\in \BZ}a^{[p]}(pj) x^{-(j+1)p}.
\end{split}
\end{equation}
Combining \eqref{eq:L-module-property02} with \eqref{eq:L-module-property01},
we get
\begin{equation}\label{enew}
Y_{W}((a(-1)^p-a^{[p]}(-p)-\chi(a)^p)\1,x)
=\sum_{j\in\BZ} a(j)^p x^{(-j-1)p}-\sum_{j\in \BZ}a^{[p]}(pj) x^{-(j+1)p}-\chi(a)^p 1_W.
\end{equation}

If $W$ is a $V^{\chi}_{\hat\fg}(\ell,0)$-module,
as $(a(-1)^p-a^{[p]}(-p)-\chi(a)^p)\1\in J_\chi$, the left-hand side of (\ref{enew}) vanishes. Thus
\begin{equation*}
    \sum_{j\in\BZ} a(j)^p x^{(-j-1)p}-\sum_{j\in \BZ}a^{[p]}(pj) x^{-(j+1)p}-\chi(a)^p 1_W=0,
\end{equation*}
which amounts to
$a(j)^p-a^{[p]}(jp)-\delta_{j,-1}\chi(a)^p=0$ on $W$ for all $j\in\BZ$, as desired.
The converse is also clear from Proposition \ref{Jchi}.
\end{proof}

Recall the ideal $J_{0}$ of $V_{\hat\fg}(\ell,0)$ from Proposition \ref{Jchi}.
Notice that $J_{0}$ is a graded  ideal. Then the quotient vertex algebra $V^0_{\hat\fg}(\ell,0)$ is
naturally a $\BZ$-graded vertex algebra with
\begin{equation}
V^{0}_{\hat{\fg}}(\ell,0)_{(n)}=0\   \mbox{ for }n<0, \   \  V^{0}_{\hat{\fg}}(\ell,0)_{(0)}=\BF \1, \mbox{ and }
V^{0}_{\hat{\fg}}(\ell,0)_{(1)}=\fg.
\end{equation}

\section{Zhu algebra $A(V^0_{\hat\fg}(\ell,0))$ and irreducible $V^0_{\hat\fg}(\ell,0)$-modules}
In this section, we continue to study the quotient vertex algebra $V^0_{\hat\fg}(\ell,0)$.
We show that the Zhu algebra $A(V^0_{\hat\fg}(\ell,0))$ is isomorphic to the restricted enveloping algebra
of the given restricted Lie algebra $\fg$ and then we classify all irreducible $\BN$-graded $V^0_{\hat\fg}(\ell,0)$-modules.

First, we consider the Zhu algebra $A(V_{\hat\fg}(\ell,0))$.
Following the same proof of Frenkel-Zhu for characteristic zero (see \cite{FZ}) we have:

\begin{proposition}\label{th:AV-aff}
The linear map
\begin{equation}\label{map-phi}
\phi: \  \fg\rightarrow A(V_{\hat\fg}(\ell,0));\ a\mapsto [a]
\end{equation}
extends uniquely to an algebra isomorphism from $U(\fg)$ to $A(V_{\hat\fg}(\ell,0))$.
\end{proposition}

Recall that $V_{\hat\fg}^{0}(\ell,0)=V_{\hat\fg}(\ell,0)/J_{0}$, where $J_{0}$ is
the $\hat\fg$-submodule of $V_{\hat\fg}(\ell,0)$
generated by vectors
\begin{equation*}
    \big(a(-m)^p-a^{[p]}(-pm)\big)\1
\end{equation*}
for $a\in \fg,\ m\in\BZ_+$. Since $a(-m)^p-a^{[p]}(-mp)$ lie in the center of $U(\hat{\fg})$, we have
\begin{equation}
  J_{0}={\rm span}\left\{  \big(a(-m)^p-a^{[p]}(-pm)\big)u\mid
a\in \fg,\ m\in\BZ_+, \  u\in V_{\hat\fg}(\ell,0)\right\}.
\end{equation}

\begin{lemma}\label{initial-step}
For $a\in\fg$, $n\in\BZ_+$, $u\in V_{\hat\fg}(\ell,0)$, the following holds in $A(V_{\hat\fg}(\ell,0))$:
\begin{equation*}
    [(a(-n)^p-a^{[p]}(-np))u]=(-1)^{n-1}[u]\big([a]^p-[a^{[p]}]\big)
\end{equation*}
and we have $[a(-n)^p\1-a^{[p]}(-np)\1]=(-1)^{n-1}\phi(a^p-a^{[p]})$.
\end{lemma}

\begin{proof}  Let $a\in \fg\subset V_{\hat\fg}(\ell,0)$.
As $\deg a=1$,  for any $k\ge 1$ and $v\in V_{\hat\fg}(\ell,0)$, (recalling that $a_{m}=a(m)$ for $m\in \BZ$) we have
\begin{equation*}
    (a(-k-1)+a(-k))v=\Res_{z}\frac{(1+z)^{\deg a}}{z^{k+1}}Y(a,z)v\in O(V_{\hat\fg}(\ell,0)),
\end{equation*}
which implies that
$[a(-k-1)v]=-[a(-k)v] $ holds in $A(V_{\hat\fg}(\ell,0))$.
Using this we get
\begin{eqnarray}\label{estep1}
[a(-n)v]=(-1)^{n-1}[a(-1)v]
\end{eqnarray}
for any $n\ge 1,\ v\in V_{\hat\fg}(\ell,0)$.

Recall from Lemma \ref{th:DnvmodO} that  $[\D^{(k)}u]= (-1)^{k}\binom{-\deg u}{k}[u]$ for any $k\in \BN$ and for
 any homogeneous $u$.
Using this fact and the skew symmetry we get
\begin{align*}
[a(-1)v]&=\sum_{k\ge 0}(-1)^{k}[\D^{(k)}v_{-1+k}a]
=\sum_{k\ge 0}(-1)^{k}\binom{-(\deg v+1-k)}{k}[v_{-1+k}a]\\
&=\sum_{k\ge 0}\binom{\deg u}{k}[v_{-1+k}a]=[v*a]=[v][a],
\end{align*}
recalling the definition of $v*a$ from (\ref{a*b}).
Consequently, we have
\begin{equation}\label{estep2}
[a(-n)v]=(-1)^{n-1}[v][a]
\end{equation}
for any $n\ge 1,\ v\in V_{\hat\fg}(\ell,0)$.
Then
\begin{equation*}
    [a(-n)^p u]=(-1)^{(n-1)p}[u][a]^p,\   \   \   \   [a^{[p]}(-np)u]=(-1)^{np-1}[u][a^{[p]}].
\end{equation*}
Noticing that $(-1)^{(n-1)p}=(-1)^{np-1}=(-1)^{n-1}$ as $p$ is odd, we obtain
\begin{eqnarray}\label{first1}
 [(a(-n)^p-a^{[p]}(-np))u]=(-1)^{n-1}[u]([a]^p-[a^{[p]}]),
\end{eqnarray}
confirming the first assertion.
The second assertion follows immediately.
\end{proof}

Recall that the restricted universal enveloping algebra $\mathfrak{u}(\fg)$ of the restricted Lie algebra $\fg$
is  the quotient algebra of $U(\fg)$ by the ideal generated by (central) elements $a^p-a^{[p]}$ for all $a\in\fg$.

As one of the main results of this section, we have:

\begin{theorem}\label{main-affinevoa}
The algebra isomorphism $\phi$ from $U(\fg)$ to $A(V_{\hat\fg}(\ell,0))$ in Proposition \ref{th:AV-aff}
reduces to an isomorphism from $\mathfrak{u}(\fg)$ to $A(V^0_{\hat\fg}(\ell,0))$.
\end{theorem}

\begin{proof}
Set $[J_0]=\{ [u]\mid u\in J_0\}\subset A(V_{\hat\fg}(\ell,0))$. From \cite{FZ}, $[J_0]$ is an ideal of
$A(V_{\hat\fg}(\ell,0))$ and
\begin{equation*}
    A(V^0_{\hat\fg}(\ell,0))\simeq A(V_{\hat\fg}(\ell,0))/[J_0].
\end{equation*}
By Lemma \ref{initial-step} we see that $[J_0]$ is exactly the ideal of $A(V_{\hat\fg}(\ell,0))$ generated by
elements $[a]^p-[a^{[p]}]$ for $a\in\fg$. Then it follows that the algebra isomorphism $\phi$ from
$U(\fg)$ to $A(V_{\hat\fg}(\ell,0))$ in Proposition \ref{th:AV-aff}
reduces to an isomorphism from $\mathfrak{u}(\fg)$ to $A(V^0_{\hat\fg}(\ell,0))$.
\end{proof}

Let $U$ be a general $\fg$-module. We make $U$ a $(\fg[t]+\BF \bk)$-module by
letting $t\fg [t]$ act on $U$ trivially and $\bk$ act as scalar $\ell$. Form a generalized Verma module
\begin{equation*}
    M_{\hat{\fg}}(\ell,U)=U(\hat{\fg})\otimes _{U(\fg[t]+\BF \bk)}U,
\end{equation*}
which is a $\BZ$-graded $\hat{\fg}$-module of level $\ell$. We have
\begin{equation*}
    M_{\hat{\fg}}(\ell,U)=\bigoplus_{n\in \BZ}M_{\hat{\fg}}(\ell,U)(n),
\end{equation*}
where $M_{\hat{\fg}}(\ell,U)(n)=0$ for $n<0$ and $M_{\hat{\fg}}(\ell,U)(0)=U$.
That is, $M_{\hat{\fg}}(\ell,U)$ is an $\BN$-graded module.
Denote by $L_{\hat{\fg}}(\ell,U)$ the quotient of
$M_{\hat{\fg}}(\ell,U)$ by the (unique) maximal graded submodule $W$ with $W(0)=0$.
Then $L_{\hat{\fg}}(\ell,U)$ is also an $\BN$-graded module with
$L_{\hat{\fg}}(\ell,U)(0)=U$. It can be readily seen that if $U$ is an irreducible $\fg$-module,
then $L_{\hat{\fg}}(\ell,U)$ is  an irreducible $\BN$-graded module.
Notice that $L_{\hat{\fg}}(\ell,U)$ is naturally an irreducible $\BN$-graded $V_{\hat\fg}(\ell,0)$-module.

Combining Propositions \ref{WtoA(V)module} and \ref{A(V)toV} with Theorem \ref{main-affinevoa}
we immediately have:

\begin{corollary}\label{V0irreducible}
For any irreducible $\mathfrak{u}(\fg)$-module $U$, $L_{\hat{\fg}}(\ell,U)$ is naturally an irreducible
$\BN$-graded $V^0_{\hat\fg}(\ell,0)$-module, and on the other hand, every irreducible
$\BN$-graded $V^0_{\hat\fg}(\ell,0)$-module is isomorphic to a module of this form.
\end{corollary}

Recall that vertex algebra $V_{\hat\fg}(\ell,0)$ is a $\BZ$-graded $\hat{\fg}$-module with
$V_{\hat\fg}(\ell,0)_{(n)}=0$ for $n<0$ and $V_{\hat\fg}(\ell,0)_{(0)}=\BF \1$, where $V_{\hat\fg}(\ell,0)$
as a $\hat{\fg}$-module is generated from $\1$.
Let $J$ be the maximal graded $\hat{\fg}$-submodule of $V_{\hat\fg}(\ell,0)$ with $J_{(0)}=0$.
Recall that a left ideal of vertex algebra $V_{\hat\fg}(\ell,0)$ exactly amounts to
a $\hat{\fg}$-submodule of $V_{\hat\fg}(\ell,0)$.
By Lemma \ref{max-ideal},  $J$ is the (unique) maximal graded ideal of vertex algebra $V_{\hat\fg}(\ell,0)$.

Set
\begin{equation}
L_{\hat\fg}(\ell,0)=V_{\hat\fg}(\ell,0)/J,
\end{equation}
a simple $\BN$-graded vertex algebra with $L_{\hat\fg}(\ell,0)_{(0)}=\BF \1$. Furthermore, we have
$L_{\hat\fg}(\ell,0)_{(1)}=\fg$ if $\ell\ne 0$ and if the invariant bilinear form  $\langle\cdot,\cdot\rangle$ on $\fg$
is non-degenerate, just as in the case of characteristic $0$ (cf. \cite{LL}).
It is clear that  $L_{\hat\fg}(\ell,0)$ is naturally a quotient vertex algebra of $V^0_{\hat\fg}(\ell,0)$.

From Corollary \ref{V0irreducible} we immediately have:

\begin{corollary}
Every $L_{\hat\fg}(\ell,0)$-module as a $V_{\hat\fg}(\ell,0)$-module
 or a $\hat{\fg}$-module is isomorphic to $L_{\hat\fg}(\ell,U)$ for
some irreducible $\mathfrak{u}(\fg)$-module $U$.
\end{corollary}

On the other hand, we have:

\begin{proposition}
Assume that $\fg$ is finite-dimensional.
For any $\ell\in \BF$, vertex algebra $V^0_{\hat\fg}(\ell,0)$ has only finitely many
irreducible $\BN$-graded modules up to equivalences and
$V^0_{\hat\fg}(\ell,0)$ is $C_2$-cofinite.
\end{proposition}

\begin{proof} The first assertion is clear from Corollary \ref{V0irreducible}
as $\mathfrak{u}(\fg)$ is a finite-dimensional algebra.  For convenience, here
denote $V_{\hat\fg}(\ell,0)$ and $V^{0}_{\hat\fg}(\ell,0)$ simply by $V$ and $V^0$, respectively.
From Lemma \ref{symmetric-algebra},  the commutative associative algebra $V/C_2(V)$ is generated by
$a+C_2(V)$ for $a\in \fg$.
Then $V/C_2(V)$ is isomorphic to a quotient of the symmetric algebra $S(\fg)$.
(In fact, one can show that $V/C_2(V)$ is isomorphic to $S(\fg)$.)
Furthermore, for $a\in \fg$, as $a(-1)^{p}\1-a^{[p]}(-p)\1=0$ in $V^0$ and
$a^{[p]}(-p)\1\in C_{2}(V^0)$, we have $a^{p}=0$ in $V^{0}/C_{2}(V^0)$.
It then follows that $V^{0}/C_{2}(V^0)$ is finite-dimensional.
\end{proof}

For the rest of this  section, we assume that $\fg$ is a finite-dimensional classical simple Lie algebra
with a non-degenerate symmetric invariant bilinear form $\langle\cdot,\cdot\rangle$.
It was known (see \cite{Jac2}, \cite{BW1, BW2}) that $\fg$ is a restricted Lie algebra with a uniquely determined
$p$-mapping.
 Let $h^{\vee}$ denote the dual Coxeter number of $\fg$.
Just as in the case of characteristic zero (see \cite{FZ}), using the Segal-Sugawara construction
we have:

\begin{proposition}
Let $\ell\in \BF$ be such that $\ell\ne -h^{\vee}$. Then there is a conformal vector $\omega \in V_{\hat{\fg}}(\ell,0)$
that makes $V_{\hat{\fg}}(\ell,0)$ a vertex operator algebra with the $\BZ$-grading given in (\ref{Zgrading-affine}).
\end{proposition}

Note that in the case of characteristic $0$, if $\ell$ is generic, then $V_{\hat\fg}(\ell,0)$ is simple, i.e.,
$V_{\hat\fg}(\ell,0)=L_{\hat\fg}(\ell,0)$. Here, we believe that the following is true:

\begin{conjecture}
Assume that $\BF$ is algebraically closed and $\ell\notin \BF_{p}=\{0,1,\dots,p-1\}$. Then
$V^0_{\hat\fg}(\ell,0)=L_{\hat\fg}(\ell,0)$ is a simple $\BZ$-graded vertex algebra.
\end{conjecture}

Consider the simplest case with $\fg={\mathfrak{sl}}_{2}$. Let $\{e,f,h\}$ be the standard basis with
\begin{equation*}
    [e,f]=h,\   \   \   [h,e]=2e,\   \   \  [h,f]=-2f.
\end{equation*}
It is straightforward to show that
\begin{equation}
(\ad e)^{3}=0,\   \   \   (\ad f)^{3}=0,  \   \   \   (\ad h)^{p}=\ad h,
\end{equation}
where we use the Fermat's little theorem.
As $p\ge 3$ by assumption, we have
\begin{equation}
(\ad e)^{p}=0,\   \   \   (\ad f)^{p}=0,  \   \   \   (\ad h)^{p}=\ad h.
\end{equation}
Note that the Killing form on ${\mathfrak{sl}}_{2}$ is non-degenerate. By Theorem \ref{p-mapping-sf}
(\cite{Jac2}; Theorem 11, Corollary) we have
\begin{equation}\label{sl2-pmapping}
e^{[p]}=0,\   \   \   f^{[p]}=0,\  \  \   h^{[p]}=h.
\end{equation}
Then $\mathfrak{u}({\mathfrak{sl}}_2)$ is the quotient algebra of $U({\mathfrak{sl}}_2)$ modulo relations
\begin{equation}
e^{p}=0,\   \   \    \    f^{p}=0,\   \   \   \   h^{p}-h=0.
\end{equation}
It follows that  the highest weight modules $L(\lambda)$  with $\lambda=0,1,\dots,p-1$
are all the irreducible $\mathfrak{u}({\mathfrak{sl}}_2)$-modules
when $p>2$ (cf. \cite{hum} (page 34, Exer. 5), \cite{SF98}), up to isomorphisms.

On the other hand, from Proposition \ref{Jchi} the ideal $J_{0}$ of
$V_{\hat\fg}(\ell,0)$ is generated by vectors
\begin{equation}
e(-1)^{p}\1,\    \    \   f(-1)^{p}\1,\    \   \   (h(-1)^{p}-h(-p))\1.
\end{equation}

\begin{lemma}\label{nilpotent-ell}
Let $\ell\in \{ 0,1,\dots,p-1\}$.  Then
\begin{equation}
e(-1)^{\ell+1}\1=0,\    \    \    \ f(-1)^{\ell+1}\1=0\    \  \mbox{ in }L_{\hat\fg}(\ell,0).
\end{equation}
\end{lemma}

\begin{proof} We need to prove that  the maximal graded ideal $J$ contains $e(-1)^{\ell+1}\1$ and $f(-1)^{\ell+1}\1$.
In view of Lemma \ref{singular-ideal}, it suffices to show that $e(-1)^{\ell+1}\1,\  f(-1)^{\ell+1}\1 \in \Omega(V_{\hat{\fg}}(\ell,0))$.
Just as in the case of characteristic $0$ (see \cite{Kac}),  one can show that
$e(-1)^{\ell+1}\1$ and $f(-1)^{\ell+1}\1$ are singular vectors in $V_{\hat\fg}(\ell,0)$ in the sense that
they are annihilated by $e(0)$ and $f(1)$, or equivalently, by $e(0)$ and $t\fg[t]$.
Then for $w=e(-1)^{\ell+1}\1,\  f(-1)^{\ell+1}\1$, we have $a(n)w=0$ for $a\in \fg,\ n\ge 1$.
Let $w$ be any vector satisfying the condition that $a(n)w=0$ for $a\in \fg,\ n\ge 1$. We claim that $w\in \Omega(V_{\hat{\fg}}(\ell,0))$.
For this, we show by induction on $n$  that $u_{m}w=0$ for any homogeneous vector $u$ of degree $n$ and for any $m\ge \deg u$.
This is true for $n=0,1$ as $V_{\hat{\fg}}(\ell,0)_{(0)}=\BF \1$ and $V_{\hat{\fg}}(\ell,0)_{(1)}=\fg$.
Assume that $n$ is a positive integer such that for every homogeneous vector $u$ with $\deg u\le n$,
$u_{m}w=0$ for all $m\ge \deg u$.
Let $v$ be a homogeneous vector of degree $n+1$.
Then $v$ is a linear combination of vectors $a(-k)u$ where $a\in \fg,\ k\ge 1$ and $u$ is a homogeneous vector with $\deg u\le n$.
 For $m\in \BZ$, we have
\begin{equation}
(a(-k)u)_{m}w=\sum_{i\ge 0}\binom{-k}{i}(-1)^{i}\left(a(-k-i)u_{m+i}w-(-1)^{k}u_{m-k-i}a(i)w\right).
\end{equation}
Assume $m\ge \deg (a(-k)u)$. That is, $m\ge \deg u+k$.
Then we have $m+i\ge \deg u$, so $u_{m+i}w=0$.
On the other hand, we have $a(i)w=0$ for $i\ge 1$ by assumption. As for $i=0$,  we have
\begin{equation*}
    u_{m-k}a(0)w=a(0)u_{m-k}w-(a_{0}u)_{m-k}w=0,
\end{equation*}
noticing that $m-k\ge \deg u$ and $\deg a_{0}u=\deg u\le n$. Thus $(a(-k)u)_{m}w=0$ for all $m\ge \deg (a(-k)u)$.
It follows that $w\in \Omega(V_{\hat{\fg}}(\ell,0))$. This proves $e(-1)^{\ell+1}\1,\  f(-1)^{\ell+1}\1 \in \Omega(V_{\hat{\fg}}(\ell,0))$.
By Lemma \ref{singular-ideal}, $e(-1)^{\ell+1}\1$ and $f(-1)^{\ell+1}\1$ generate a proper graded ideal of
$V_{\hat\fg}(\ell,0)$.
Then $J$ contains $e(-1)^{\ell+1}\1$ and $f(-1)^{\ell+1}\1\in J.$ This completes the proof.
\end{proof}

Note that $L_{\hat{\fg}}(\ell,0)$ is a quotient vertex algebra of $V^0_{\hat\fg}(\ell,0)$, so that
$A(L_{\hat{\fg}}(\ell,0))$ is a quotient algebra of $A(V^0_{\hat\fg}(\ell,0))$.
In view of Lemma \ref{nilpotent-ell}, we also have
$[e]^{\ell+1}=0=[f]^{\ell+1}$ in $A(L_{\hat{\fg}}(\ell,0))$.
Consequently,  $L_{\hat{\fg}}(\ell,0)$ has at most $\ell+1$ irreducible $\BN$-graded modules
up to equivalences. The following is our belief:

\begin{conjecture}\label{nonnegative-level}
For $\ell\in \{0,1,\dots,p-1\}$, $L_{\hat{\fg}}(\ell,0)$ has exactly $\ell+1$ irreducible $\BN$-graded modules
up to equivalences.
\end{conjecture}

If $\ell\in \{0,1,\dots,p-2\}$, since $e(-1)^{\ell+1}\1$ and $f(-1)^{\ell+1}\1$ are of degree $\ell+1<p$ and nonzero
in $V^0_{\hat\fg}(\ell,0)$, we have $V^0_{\hat\fg}(\ell,0)\ne L_{\hat\fg}(\ell,0)$.
If Conjecture \ref{nonnegative-level} is true, then $V^0_{\hat\fg}(p-1,0)$ and $L_{\hat\fg}(p-1,0)$
have the same irreducible $\BN$-graded modules. In view of this,
we believe the following is true:

\begin{conjecture}
$V^0_{\hat\fg}(p-1,0)=L_{\hat\fg}(p-1,0)$.
\end{conjecture}

\appendix\setcounter{section}{1}\setcounter{theorem}{0}\setcounter{equation}{0}
\section*{Appendix: Certain identities on binomial coefficients}

\begin{lemma}\label{simple-facts}
For $m,n\in \BZ,\ k\in \BN$, we have
\begin{gather}
    (m-n)\binom{m+n+1}{k}=\sum_{i=0}^k(m-n-k+2i)\binom{m+1}{k-i}\binom{n+1}{i}, \label{eq:simple001}\\
    \sum_{i=0}^k\binom{m+1}{k-i}\binom{n+1}{i}\binom{m-k+i+1}{3}\delta_{m+n-k,0}=\binom{m+1}{3}\delta_{m+n,0}\delta_{k,0} \label{eq:simple002}
\end{gather}
\end{lemma}

\begin{proof}
If $k=0$, both sides of \eqref{eq:simple001} are equal to $m-n$.
Assume $k\ge 1$. Then
\begin{align*}
    &\phantom{=}\;\;\sum_{i=0}^k(m-n-k+2i)\binom{m+1}{k-i}\binom{n+1}{i}\\
    &=\sum_{i=0}^k (m-n)\binom{m+1}{k-i}\binom{n+1}{i}-\sum_{i=0}^k (k-i)\binom{m+1}{k-i}\binom{n+1}{i}\\
    &\qquad +\sum_{i=0}^k i\binom{m+1}{k-i}\binom{n+1}{i}\\
    &=(m-n)\binom{m+n+2}{k}-\sum_{i=0}^{k-1} (m+1)\binom{m}{k-i-1}\binom{n+1}{i}\\
    &\qquad +\sum_{i=1}^{k} (n+1)\binom{m+1}{k-i}\binom{n}{i-1}\\
    &=(m-n)\binom{m+n+2}{k}-(m+1)\binom{m+n+1}{k-1}+(n+1)\binom{m+n+1}{k-1}\\
    &=(m-n)\left(\binom{m+n+2}{k}-\binom{m+n+1}{k-1}\right)\\
    &=(m-n)\binom{m+n+1}{k}.
\end{align*}
This proves that \eqref{eq:simple001} holds for all $k\in \BN$.

On the other hand, we have
\begin{align*}
    &\phantom{=}\,\,\sum_{i=0}^k\binom{m+1}{k-i}\binom{n+1}{i}\binom{m-k+i+1}{3}\delta_{m+n-k,0}\\
    &=\sum_{i=0}^k \binom{m+1}{3} \binom{m-2}{k-i}\binom{n+1}{i}\delta_{m+n-k,0}\\
    &=\binom{m+1}{3} \binom{m+n-1}{k}\delta_{m+n-k,0}\\
    &=\binom{m+1}{3} \binom{k-1}{k}\delta_{m+n-k,0}\\
    &=\binom{m+1}{3}\delta_{m+n,0}\delta_{k,0},
\end{align*}
proving \eqref{eq:simple002}.
\end{proof}

\end{document}